\documentclass[12pt, reqno]{amsart}
\usepackage{amssymb, amsthm, amsmath, amsfonts, mathrsfs}
\usepackage{array, epsfig}
\usepackage{bbm}
\usepackage{hyperref}
\usepackage{color}

\usepackage[left=2.2cm, top=2.45cm,bottom=2.45cm,right=2.2cm]{geometry}

\newcommand{\dist}{\mathop{\mathrm{dist}}\nolimits}
\newcommand{\intt}{\operatorname{int}}
\newcommand{\cl}{\operatorname{cl}}

\newcommand{\bE}{\mathbb{E}}

\newcommand{\bN}{\mathbb{N}}

\newcommand{\cD}{\mathcal{D}}

\newcommand{\cL}{\mathcal{L}}

\newcommand{\cN}{\mathcal{N}}
\newcommand{\cV}{\mathcal{V}}
\newcommand{\cW}{\mathcal{W}}

\newcommand{\Beta}{\text{\rm Beta}}

\newcommand{\E}{\mathbb E}

\newcommand{\R}{\mathbb{R}}
\newcommand{\N}{\mathbb{N}}

\newcommand{\C}{\mathbb{C}}

\renewcommand{\P}{\mathbb{P}}

\renewcommand{\Re}{\operatorname{Re}}

\newcommand{\Var}{\mathop{\mathrm{Var}}\nolimits}

\newcommand{\aff}{\mathop{\mathrm{aff}}\nolimits}

\newcommand{\eps}{\varepsilon}
\newcommand{\eqdistr}{\stackrel{d}{=}}

\newcommand{\todistr}{\overset{d}{\underset{n\to\infty}\longrightarrow}}

\newcommand{\toas}{\overset{a.s.}{\underset{n\to\infty}\longrightarrow}}

\newcommand{\ton}{\overset{}{\underset{n\to\infty}\longrightarrow}}

\newcommand{\dd}{{\rm d}}
\newcommand{\eee}{{\rm e}}

\theoremstyle{plain}
\newtheorem{theorem}{Theorem}[section]
\newtheorem{lemma}[theorem]{Lemma}
\newtheorem{corollary}[theorem]{Corollary}
\newtheorem{proposition}[theorem]{Proposition}

\theoremstyle{definition}

\theoremstyle{remark}
\newtheorem{remark}[theorem]{Remark}

\begin{document}

\author{Julian Grote}
\address{Julian Grote: Fakult\"at f\"ur Mathematik,
Ruhr-Universit\"at Bochum,
44780 Bochum, Germany}
\email{julian.grote@rub.de}

\author{Zakhar Kabluchko}
\address{Zakhar Kabluchko: Institut f\"ur Mathematische Stochastik,
Universit\"at M\"unster,
48149 M\"unster, Germany}
\email{zakhar.kabluchko@uni-muenster.de}

\author{Christoph Th\"ale}
\address{Chrisotph Th\"ale: Fakult\"at f\"ur Mathematik,
Ruhr-Universit\"at Bochum,
44780 Bochum, Germany}
\email{christoph.thaele@rub.de}

\title[Limit theorems for random simplices]{Limit theorems for random simplices in high dimensions}


\begin{abstract}
Let $r=r(n)$ be a sequence of integers such that $r\leq n$ and let $X_1,\ldots,X_{r+1}$ be independent random points distributed according to the Gaussian, the Beta or the spherical distribution on $\mathbb{R}^n$. Limit theorems for the log-volume and the volume of the random convex hull of $X_1,\ldots,X_{r+1}$ are established in high dimensions, that is, as $r$ and $n$ tend to infinity simultaneously. This includes, Berry-Esseen-type central limit theorems, log-normal limit theorems, moderate and large deviations. Also different types of mod-$\phi$ convergence are derived. The results heavily depend on the asymptotic growth of $r$ relative to $n$. For example, we prove that the fluctuations of the volume of the simplex are normal (respectively, log-normal) if $r=o(n)$ (respectively, $r\sim \alpha n$ for some $0 < \alpha < 1$). 
\end{abstract}

\maketitle

\section{Introduction}

In the last decades, random polytopes have become one of the most central models studied in stochastic geometry. In particular, they have seen numerous applications to other branches of mathematics such as asymptotic geometric analysis, compressed sensing, computational geometry, optimization or multivariate statistics; see, for example, the surveys of B\'ar\'any \cite{BaranySurvey}, Hug \cite{HugSurvey} and Reitzner \cite{ReitznerSurvey} for further details and references. The focus in most works has been on models of the following type. First, we fix a space dimension $n\in\N$ and a probability measure $\mu$ on $\R^n$. Then, we let $X_1,\ldots,X_r$, where $r\geq n+1$, be independent random points in $\R^n$ that are distributed according to $\mu$. A random polytope $P_r$ now arises by taking the convex hull of the point set $X_1,\ldots,X_r$. Starting with the seminal paper of R\'enyi and Sulanke \cite{RenyiSulanke}, the asymptotic behaviour of the expectation and the variance of the volume or the number of faces of $P_r$ has been studied intensively, as $r\to\infty$, while keeping $n$ fixed. Moreover, it has been investigated whether these quantities satisfy a 'typical' and 'atypical' behaviour, i.e., fulfil a central limit theorem, large or moderate deviation principles and concentration inequalities, respectively, to name just a few topics of current research.

However, up to a few exceptions it has not been investigated what happens if the space dimension $n$ is not fixed, but tends to infinity. The only such exceptions we were able to localize in the literature are the papers of Ruben \cite{RubenCLT}, Mathai \cite{MathaiCLT}, Anderson \cite{Anderson} and Maehara \cite{Maehara}. It is shown in the first two of these works that for any {\it fixed} $r\in\N$ the $r$-volume of the convex hull of $r+1\le n+1$ independent and uniform random points of which some are in the interior of the $n$-dimensional unit ball and the others on its boundary, is asymptotically normal, as $n\to\infty$. The third one establishes the same result in the situation where the $r$ points are distributed according to the so-called Beta-type distribution in the $n$-dimensional unit ball. The fourth mentioned text generalizes the set-up to an arbitrary underlying $n$-fold product distribution on $\mathbb{R}^n$.

On the other hand, the regime in which $r$ and $n$ tend to infinity \textit{simultaneously} is not treated in these papers. The purpose of the present text is to close this gap and to prove a collection of probabilistic limit theorems for the $r$-volume of the convex hull of $r+1\le n+1$ random points that are distributed according to certain classes of probability distributions that allow for explicit computations, especially focusing on different regimes of growths of the parameter $r$ relative to $n$. More precisely, we distinguish between the following three regimes. The first one is the case where $r$ grows like $o(n)$ with the dimension $n$, which means that $r/n$ converges to zero, as $n \rightarrow \infty$. This of course includes the situations where $r$ is fixed -- covering thereby the case considered in the four papers mentioned above -- or behaves like $n^\alpha$ wir $\alpha\in (0,1)$, to give just two examples (let us emphasize at this point that we interpret expressions like $\sqrt{n}$ or $n/2$ as $\lfloor \sqrt{n} \rfloor$ and $\lfloor n/2 \rfloor$, respectively, in what follows). Secondly, the underlying situation might be the one where $r$ is asymptotically equivalent to $\alpha n$ for some $\alpha \in (0,1)$. Lastly, we analyse the setting where $n-r = o(n)$, as $n\rightarrow \infty$. In particular, for $r=n$ we arrive in the situation where we choose $n+1$ random points and thus their convex hull is nothing but a full-dimensional simplex in $\mathbb{R}^n$.\\

Our paper and the results we are going to present (and which represent a 'complete' description of the high-dimensional probabilistic behaviour of the underlying random simplices) are organized as follows. In Section $2$ we introduce the different random point models we consider and state formulas for the moments of the volume of the random simplices induced by the convex hulls of these point sets. By using these moments, we are then able to derive the precise distributions of the previously mentioned volumes. In Section $3$ we start with the first limit theorems. By using the method of cumulants we give 'optimal' Berry-Esseen bounds and moderate deviation principles for the logarithmic volumes of our random simplices. Then, we transfer the limit theorem from the log-volume to the volume itself and obtain thereby a phase transition in the limiting behaviour depending on the choice of the parameter $r$. Section $4$ establishes results concerning so-called mod-$\phi$ convergence and is also the starting point to prove the results presented in Section $5$, where we add large deviation principles to the moderate ones obtained earlier in Section $3$.

\section{Models, volumes and probabilistic representations}

\subsection{The four models}\label{sec:SectionModels}
In this paper we consider convex hulls of random points $X_1,X_2,\ldots$ We only consider the following four models which allow for explicit computations. These models were identified in~\cite{Miles71} and~\cite{RubenMiles80} by Miles and Ruben and Miles, respectively.

\begin{itemize}
\item [(a)] The \emph{Gaussian model}: $X_1,X_2,\ldots$ are i.i.d.\ with standard normal density
$$
f(|x|) = (2\pi)^{-n/2} \cdot \eee^{-\frac 12 |x|^2}, \quad x\in\R^n.
$$
\item [(b)] The \emph{Beta model} with parameter $\nu>0$: $X_1,X_2,\ldots$ are i.i.d.\ points in the ball of radius $1$ with density
$$
f(|x|) = \frac 1 {\pi^{n/2}} \frac{\Gamma\left(\frac{n+\nu}{2}\right)}{\Gamma\left(\frac \nu 2\right)} \cdot \left(1- |x|^2\right)^{(\nu-2)/2}, \quad x\in\R^n, \;\; |x| < 1.
$$
\item [(c)] The \emph{Beta prime model} with parameter $\nu>0$: $X_1,X_2,\ldots$ are i.i.d.\ points with density
$$
f(|x|) = \frac 1 {\pi^{n/2}} \frac{\Gamma\left(\frac{n+\nu}{2}\right)}{\Gamma\left(\frac \nu 2\right)} \cdot \left(1 + |x|^2\right)^{-(n+ \nu)/2}, \quad x\in\R^n.
$$
\item [(d)] The \emph{spherical model}: $X_1,X_2,\ldots$ are uniformly distributed on the sphere of radius $1$ centered at the origin of $\R^n$.
\end{itemize}

\begin{remark}
Observe that in the Beta prime model the power is $(n+\nu)/2$ (which depends on $n$) rather than just $\nu/2$.
\end{remark}

\subsection{Moments for the volumes of random simplices and parallelotopes}
Let $1\leq r\leq n$ be an integer and $X_1,\ldots,X_{r+1}$ be independent random points in $\R^n$ that are distributed according to one of the distributions introduced in Section \ref{sec:SectionModels}. By $\mathcal V_{n,r}$ we denote the $r$-dimensional volume of the simplex with vertices $X_1,\ldots,X_{r+1}$. Moreover, we use the symbol $\mathcal W_{n,r}$ to indicate the $r$-dimensional volume of the parallelotope spanned by the vectors $X_1,\ldots,X_r$. Note that up to a factor $r!$, $\cW_{n,r}$ is the same as the $r$-dimensional volume of the simplex with vertices $0,X_1,\ldots,X_{r}$. We start by recalling formulas for the moments of $\cW_{n,r}$. Moments of integer orders can directly be computed using the well-known linear Blaschke-Petkantschin formula from integral geometry together with an induction argument.

\begin{theorem}[Moments for parallelotopes]\label{theo:vol_Parallelotopes_Moments}
Let $\mathcal W_{n,r}$ be the volume of the $r$-dimensional parallelotope spanned by the vectors $X_1,\ldots,X_{r}$ chosen according to one of the above four models.
\begin{itemize}
\item[(a)] In the Gaussian model we have, for all $k\geq 0$,
$$
\E[\cW_{n,r}^{2k}] = \prod_{j=1}^r\Bigg[2^k{\Gamma\Big({n-r+j\over 2}+k\Big)\over\Gamma\Big({n-r+j\over 2}\Big)}\Bigg].
$$
\item[(b)] In the Beta model with parameter $\nu>0$ we have, for all $k\geq 0$,
$$
\E[\cW_{n,r}^{2k}] =\prod_{j=1}^r\Bigg[{\Gamma\Big({n-r+j\over 2}+k\Big)\Gamma\Big({n+\nu\over 2}\Big)\over\Gamma\Big({n-r+j\over 2}\Big)\Gamma\Big({n+\nu\over 2}+k\Big)}\Bigg].
$$
\item[(c)] In the Beta prime model with parameter $\nu>0$ we have, for all $k\in (0,\frac \nu 2]$,
$$
\E[\cW_{n,r}^{2k}] =\prod_{j=1}^r\Bigg[{\Gamma\Big({n-r+j\over 2}+k\Big)\Gamma\Big({\nu\over 2}-k\Big)\over\Gamma\Big({n-r+j\over 2}\Big)\Gamma\Big({\nu\over 2}\Big)}\Bigg].
$$
\item[(d)] In the spherical model we have, for all $k\geq 0$,
$$
\E[\cW_{n,r}^{2k}] =\prod_{j=1}^r\Bigg[{\Gamma\Big({n-r+j\over 2}+k\Big)\Gamma\Big({n\over 2}\Big)\over\Gamma\Big({n-r+j\over 2}\Big)\Gamma\Big({n\over 2}+k\Big)}\Bigg].
$$
\end{itemize}
\end{theorem}
\begin{proof}
The formula in (a) can be concluded from \cite{Mathai_Random_Parallelotopes} or \cite{Ruben79}. Formula (b) is Theorem 19.2.5 from \cite{mathai_charalambides}, Formula (c) is Theorem 19.2.6 from \cite{mathai_charalambides}. Formula (d) is the limiting case of (c) for $\nu \downarrow 0$ but is actually also contained both in Theorems~19.2.5 and 19.2.6 from \cite{mathai_charalambides} because these deal with a slightly more general model which allows for some points to be chosen uniformly on the unit sphere.
\end{proof}

For simplices, the moments are very similar. The products appearing in the formulas for simplices are the same as for parallelotopes, but certain additional factors involving the $\Gamma$-function appear. Again, for moments of integer order, a direct proof for these formulas can be carried out using the affine Blaschke-Petkantschin formula and an induction argument (compare, for example, with the proof of \cite[Theorem 8.2.3]{SW} for the special case of the Beta model with $\nu=2$ and the spherical model.)

\begin{theorem}[Moments for simplices]\label{theo:vol_Simplices_Moments}
Let $\mathcal V_{n,r}$ be the volume of the $r$-dimensional simplex with vertices $X_1,\ldots,X_{r+1}$ chosen according to one of the above four models.
\begin{itemize}
\item[(a)] In the Gaussian model we have, for all real $k\geq 0$,
$$
\E[(r!\cV_{n,r})^{2k}] = (r+1)^{k}\, \prod_{j=1}^r\bigg[2^{k}{\Gamma\big({n-r+j\over 2}+k\big)\over\Gamma\big({n-r+j\over 2}\big)}\bigg].
$$
\item[(b)] In the Beta model with parameter $\nu>0$ we have, for all real $k\geq 0$,
\begin{align*}
\E[(r!\cV_{n,r})^{2k}]
&=
\prod_{j=1}^r\Bigg[{\Gamma\Big({n-r+j\over 2}+k\Big)\over\Gamma\Big({n-r+j\over 2}\Big)}{\Gamma\Big({n+\nu\over 2}\Big)\over\Gamma\Big({n+\nu\over 2}+k\Big)}\Bigg]\\
&\qquad\qquad\qquad\times{\Gamma\Big({n+\nu\over 2}\Big)\over\Gamma\Big({n+\nu\over 2}+k\Big)}{\Gamma\Big({r(n+\nu-2)+(n+\nu)\over 2}+(r+1)k\Big)\over\Gamma\Big({r(n+\nu-2)+(n+\nu)\over 2}+rk\Big)}.
\end{align*}
\item[(c)] In the Beta prime model with parameter $\nu>0$ we have, for all  real $0\leq k<{\nu\over 2}$,
\begin{align*}
\E[(r!\cV_{n,r})^{2k}]
&=
\prod_{j=1}^r \Bigg[
\frac{\Gamma\Big(\frac{n-r+j}{2}+k\Big)}{\Gamma\Big(\frac{n-r+j}{2}\Big)}
\frac{\Gamma\Big(\frac \nu 2 -k \Big)}{\Gamma\Big(\frac \nu2\Big)}
\Bigg]
\frac{\Gamma\Big(\frac \nu 2 -k \Big)}{\Gamma\Big(\frac \nu2\Big)}
\frac{\Gamma\Big( \frac{(r+1)\nu}{2} -rk \Big)}{\Gamma\Big( \frac{(r+1)\nu}{2} - (r+1)k \Big)}.
\end{align*}
\item[(d)] In the spherical model we have, for all real $k\geq 0$,
\begin{align*}
\E[(r!\cV_{n,r})^{2k}] &= \prod_{j=1}^r\Bigg[{\Gamma\Big({n-r+j\over 2}+k\Big)\over\Gamma\Big({n-r+j\over 2}\Big)}{\Gamma\Big({n\over 2}\Big)\over\Gamma\Big({n\over 2}+k\Big)}\Bigg]{\Gamma\Big({n\over 2}\Big)\over\Gamma\Big({n\over 2}+k\Big)}{\Gamma\Big({r(n-2)+n\over 2}+(r+1)k\Big)\over\Gamma\Big({r(n-2)+n\over 2}+rk\Big)}.
\end{align*}
\end{itemize}
\end{theorem}
\begin{proof}
Formula (a) is Equation (70) in \cite{Miles71}. Formula (b) is Equation (74) in \cite{Miles71}.  Formula (c) is Equation (72) in \cite{Miles71}. Finally, Formula (d) is obtained from (b) by letting $\nu \to 0$. Note that the formula in \cite{Miles71} contains a typo, which is corrected, for example, in \cite{Chu}. Also Miles~\cite{Miles71} considers only integer moments. Extension to non-integer moments can be found in \cite{KabluchkoTemesvariThaele}.
\end{proof}

Observe that the moments in the spherical case can be obtained from the moments in the Beta model by taking $\nu=0$ there. In fact, the uniform distribution on the sphere is the weak limit of the Beta distribution as $\nu \downarrow 0$; see the proof of Theorem~\ref{theo:distance_distr}, below. Since the proofs of our limit theorems are based on the above formulas for the moments, we may and will consider the spherical and the Beta models together, the former being the special case of the latter with $\nu=0$.
We refrain from stating the limit theorems in the Beta prime case because they seem similar to the Beta case.

\subsection{Distributions for the volumes of random simplices and parallelotopes}
The purpose of this section is to derive probabilistic representations  for the random variables $\cW_{n,r}^2$ and $\cV_{n,r}^2$ for the four models introduced in Section \ref{sec:SectionModels}. For this, we need to recall certain standard distributions.
A random variable has a Gamma distribution with shape $\alpha>0$ and scale $\lambda>0$ if its density is given by
$$
g(t) = \frac{\lambda^{\alpha}}{\Gamma(\alpha)} t^{\alpha-1} \eee^{-\lambda t}, \quad t\geq 0.
$$
Especially if $\alpha={d/2}$ for some $d\in\N$ and $\lambda={1/ 2}$, we speak about a $\chi^2$ distribution with $d$ degrees of freedom. A random variable has a Beta distribution with parameters $\alpha_1>0, \alpha_2>0$ if its density is
$$
g(t) = \frac{\Gamma(\alpha_1+\alpha_2)}{\Gamma(\alpha_1)\Gamma(\alpha_2)} t^{\alpha_1-1} (1-t)^{\alpha_2-1}, \quad t\in (0,1).
$$
Finally, a random variable has a Beta prime distribution with parameters $\alpha_1>0$, $\alpha_2>0$ if its density is
$$
g(t) = \frac{\Gamma(\alpha_1+\alpha_2)}{\Gamma(\alpha_1)\Gamma(\alpha_2)} t^{\alpha_1-1} (1+t)^{-\alpha_1-\alpha_2}, \quad t > 0.
$$
Note that the Beta prime distribution coincides, up to rescaling, with the Fisher F distribution.
We agree to denote by $\chi^2_d$, respectively $\Gamma_{\alpha, \lambda}, \beta_{\alpha_1,\alpha_2}, \beta'_{\alpha_1,\alpha_2}$,  a random variable with $\chi^2$-distribution with $d\in\N$ degrees of freedom and the Gamma, Beta or Beta prime distribution with corresponding parameters, respectively. We shall also use the notation $X\sim\Beta(\alpha_1,\alpha_2)$ or $X\sim\Beta'(\alpha_1,\alpha_2)$ to indicate that a random variable $X$ has a Beta or a Beta prime distribution with parameters $\alpha_1$ and $\alpha_2$, respectively. Also, we agree that all such variables are assumed to be independent. We recall that the moments (of real order $k\geq 0$, as long as they exist) of these distributions are given by:
$$
\E [\chi_d^{2k}] = 2^k \frac{\Gamma\Big(\frac d2 + k\Big)}{\Gamma\Big(\frac d2\Big)},
\quad
\E [\beta_{\alpha_1,\alpha_2}^k] = \frac{\Gamma(\alpha_1+\alpha_2)\Gamma(\alpha_1+k)}{\Gamma(\alpha_1) \Gamma(\alpha_1+\alpha_2+k)},
\quad
\E [(\beta_{\alpha_1,\alpha_2}')^{k}] = \frac{\Gamma(\alpha_1+k)\Gamma(\alpha_2-k)}{\Gamma(\alpha_1) \Gamma(\alpha_2)}.
$$
Using Theorem \ref{theo:vol_Parallelotopes_Moments} we first obtain probabilistic representations for the volume of random parallelotopes spanned by vectors whose distributions belong to one of the classes introduced in Section \ref{sec:SectionModels}.

\begin{theorem}[Distributions for parallelotopes]\label{theo:vol_distr_linear}
Let $\mathcal W_{n,r}$ be the volume of the $r$-dimensional parallelotope spanned by the vectors $X_1,\ldots,X_{r}$ chosen according to one of the above four models.
\begin{itemize}
\item[(a)] In the Gaussian model we have $\mathcal W_{n,r}^2 \eqdistr \prod\limits_{j=1}^{r} \chi^2_{n-r+j}$.
\item[(b)] In the Beta model we have $\mathcal W_{n,r}^2 \eqdistr   \prod\limits_{j=1}^r \beta_{{n-r+j\over 2}, {\nu + r -j\over 2}}$.
\item[(c)] In the Beta prime model we have $\mathcal W_{n,r}^2 \eqdistr  \prod\limits_{j=1}^r \beta'_{{n-r+j\over 2}, {\nu\over 2}}$.
\item[(d)] In the spherical model we have $\mathcal W_{n,r}^2 \eqdistr \prod\limits_{j=1}^r \beta_{{n-r+j\over 2}, {r -j\over 2}}$.
\end{itemize}
The random variables involved in the products are assumed to be independent.
\end{theorem}

The distribution of the volume of a random simplex generated by one of the four models is more involved and can be derived from Theorem \ref{theo:vol_Simplices_Moments}.

\begin{theorem}[Distributions for simplices]\label{theo:vol_distr_affine}
Let $\mathcal V_{n,r}$ be the volume of the $r$-dimensional simplex with vertices $X_1,\ldots,X_{r+1}$ chosen according to the one of the above four models.
\begin{itemize}
\item[(a)] In the Gaussian model we have $(r!\mathcal V_{n,r})^2 \eqdistr (r+1) \prod\limits_{j=1}^{r} \chi^2_{n-r+j}$.
\item[(b)] In the Beta model we have
$\xi(1-\xi)^r  (r!\mathcal V_{n,r})^2 \eqdistr (1-\eta)^r  \prod\limits_{j=1}^r \beta_{{n-r+j\over 2}, {\nu + r -j\over 2}}$, where $\xi,\eta\sim \text{Beta}(\frac{n+\nu}{2}, \frac{r(n+\nu-2)}{2})$ are random variables such that $\xi$ is independent of $\mathcal V_{n,r}$, while $\eta$ is independent of $\beta_{{n-r+j\over 2}, {\nu + r -j\over 2}}$, $j=1,\ldots,r$.
\item[(c)] In the Beta prime model we have $(1+\eta)^r  (r!\mathcal V_{n,r})^2 \eqdistr   \xi^{-1}(1+\xi)^{r+1}  \prod\limits_{k=1}^r \beta'_{{n-r+j\over 2}, {\nu\over 2}}$, where
$\xi,\eta\sim \text{Beta}'(\frac{\nu}{2}, \frac{r\nu}{2})$ are random variables such that $\eta$ is independent of $\mathcal V_{n,r}$, while $\xi$ is independent of $\beta'_{{n - r +j\over 2}, {\nu\over 2}}$, $j=1,\ldots,r$.
\item[(d)] In the spherical model we have $\xi(1-\xi)^r  (r!\mathcal V_{n,r})^2 \eqdistr  (1-\eta)^r  \prod\limits_{j=1}^r \beta_{{n-r+j\over 2}, {r -j\over 2}}$, where $\xi,\eta\sim \text{Beta}(\frac{n}{2}, \frac{r(n-2)}{2})$ are random variables such that $\xi$ is independent of $\mathcal V_{n,r}$, while $\eta$ is independent of $\beta_{{n-r+j\over 2}, {r -j\over 2}}$, $j=1,\ldots,r$.
\end{itemize}
\end{theorem}

\begin{proof}
The assertion in (a) follows directly from Theorem \ref{theo:vol_Simplices_Moments} (a) combined with the fact that the $k$th moment of a $\chi_{n-r+j}^2$ random variable is given by
$$
2^k{\Gamma({n-r+j\over 2}+k)\over \Gamma({n-r+j\over 2})}.
$$

To prove (b) we define $\alpha_1:={n+\nu\over 2}$ and $\alpha_2:={r(n+\nu-2)\over 2}$. Denoting by $B(x,y)=\frac{\Gamma(x)\Gamma(y)}{\Gamma(x+y)}$, $x,y>0$, the Beta function, we observe that, since $\xi,\eta\sim\Beta(\alpha_1,\alpha_2)$,
$$
\E[(1-\eta)^{rk}] = {1\over B(\alpha_1,\alpha_2)}\int_0^1 x^{\alpha_1-1}(1-x)^{\alpha_2+rk-1}\,\dd x = {B(\alpha_1,\alpha_2+rk)\over B(\alpha_1,\alpha_2)}
$$
and
$$
\E[\xi^k(1-\xi)^{rk}] = {1\over B(\alpha_1,\alpha_2)}\int_0^1 x^{\alpha_1+k-1}(1-x)^{\alpha_2+rk-1}\,\dd x = {B(\alpha_1+k,\alpha_2+rk)\over B(\alpha_1,\alpha_2)}\,.
$$
This implies that
\begin{align*}
\frac {\E[(1-\eta)^{rk}]}{\E[\xi^k(1-\xi)^{rk}]} &= \frac {B(\alpha_1,\alpha_2+rk)}{B(\alpha_1+k,\alpha_2+rk)} = {\Gamma(\alpha_1+\alpha_2+(r+1)k)\Gamma(\alpha_1)\over\Gamma(\alpha_1+k)\Gamma(\alpha_1+\alpha_2+rk)}\\
&={\Gamma\Big({r(n+\nu-2)+(n+\nu)\over 2}+(r+1)k\Big)\Gamma\Big({n+\nu\over 2}\Big)\over\Gamma\Big({n+\nu\over 2}+k\Big)\Gamma\Big({r(n+\nu-2)+(n+\nu)\over 2}+rk\Big)}
\end{align*}
and this is precisely the last factor in the formula for the moments, see Theorem \ref{theo:vol_Simplices_Moments} (b).

Next, we consider (c). Since $\xi,\eta\sim \text{Beta}'(\alpha_1, \alpha_2)$ with $\alpha_1 = \frac{\nu}{2}$ and $\alpha_2 = \frac{r\nu}{2}$, we apply the formula $\int_0^{\infty} x^{\alpha_1-1} (1+x)^{-\alpha_1-\alpha_2} \dd x = B(\alpha_1,\alpha_2)$ to obtain
$$
\E [(1+\eta)^{rk}] = \frac 1 {B(\alpha_1,\alpha_2)} \int_0^\infty x^{\alpha_1-1}(1+x)^{-\alpha_1-(\alpha_2-rk)}\dd x
= \frac{B(\alpha_1, \alpha_2 - rk)}{B(\alpha_1,\alpha_2)}
$$
and
$$
\E \Big[\xi^{-k}(1+\xi)^{(r+1)k}\Big]
=
\frac 1 {B(\alpha_1,\alpha_2)}\int_0^\infty x^{\alpha_1-k-1}(1+x)^{-\alpha_1-\alpha_2 - (r+1)k} \dd x
=
\frac{B(\alpha_1 - k, \alpha_2 - rk)}{B(\alpha_1, \alpha_2)}.
$$
It follows that
\begin{align*}
\frac{\E \Big[\xi^{-k}(1+\xi)^{(r+1)k}\Big]}{\E [(1+\eta)^{rk}]}
&=
\frac{B(\alpha_1 - k, \alpha_2 - rk)}{B(\alpha_1, \alpha_2 - rk)}=
\frac{\Gamma(\alpha_1 - k)\Gamma(\alpha_1+\alpha_2 -rk)}{\Gamma(\alpha_1+\alpha_2 - (r+1) k) \Gamma(\alpha_1)}\\
&=
\frac{\Gamma\Big(\frac \nu 2 -k \Big)}{\Gamma\Big(\frac \nu2\Big)}
\frac{\Gamma\Big( \frac{(r+1)\nu}{2} -rk \Big)}{\Gamma\Big( \frac{(r+1)\nu}{2} - (r+1)k \Big)},
\end{align*}
which is exactly the last factor in the formula for the moments given by Theorem \ref{theo:vol_Simplices_Moments} (c).

The assertion in (d) follows as a limit case from that in (b), as $\nu\downarrow 0$.
\end{proof}

\begin{remark}
The distributional equality in Theorem \ref{theo:vol_distr_affine} (a) has already been noted by Miles, see \cite[Section 13]{Miles71}. The other probabilistic representations in (b)--(d) seem to be new.
\end{remark}

\subsection{Distance distributions}
As in the previous sections let $X_1,\ldots,X_{r+1}$ be independent random points that are distributed according to one of the four models from Section \ref{sec:SectionModels}. Our interest now lies in the distance from the origin to the $r$-dimensional affine subspace spanned by $X_1,\ldots,X_{r+1}$.

\begin{theorem}[Distance distributions]\label{theo:distance_distr}
Let $X_1,\ldots,X_{r+1}$ be chosen according to one of the above four models and denote by $\cD_{n,r}$ the distance from the origin to the $r$-dimensional affine subspace spanned by $X_1,\ldots,X_{r+1}$.
\begin{itemize}
\item[(a)] In the Gaussian model we have $\cD_{n,r}^2\eqdistr(r+1)^{-1}\chi_{n-r}^2$.
\item[(b)] In the Beta model we have $\cD_{n,r}^2\eqdistr \beta_{{n-r\over 2},{\nu(r+1)+r(n-1)\over 2}}$.
\item[(c)] In the Beta prime model we have $\cD_{n,r}^2\eqdistr \beta'_{{n-r\over 2},{\nu(r+1)\over 2}}$.
\item[(d)] In the spherical model we have $\cD_{n,r}^2\eqdistr \beta_{{n-r\over 2},{r(n-1)\over 2}}$.
\end{itemize}
\end{theorem}
\begin{proof}
The density of $\cD_{n,r}$ in the cases (a)--(c) can be computed from a formula on page 16 in \cite{RubenMiles80}. In fact, for the Gaussian model we obtain that $\cD_{n,r}$ has density
$$
h\mapsto c_{n,r}\,h^{n-r-1}\,e^{-{h^2(r+1)\over 2}}\,,\qquad h>0\,,
$$
which implies (a). For the Beta model we obtain the density
$$
h\mapsto c_{n,r,\nu}\,h^{n-r-1}(1-h^2)^{{r(n+1)\over 2}+{(r+1)(\nu-2)\over 2}}\,,\qquad 0<h<1\,,
$$
for $\cD_{n,r}$ and (b) follows. Next, for the Beta prime model the density of $\cD_{n,r}$ is given by
$$
h\mapsto c_{n,r,\nu}\,h^{n-r-1}(1+h^2)^{{r(n+1)\over 2}-{(r+1)(n+\nu)\over 2}}\,,\qquad h>0\,,
$$
whence (c) follows. Finally, the spherical model follows from the Beta model in the limit, as $\nu\downarrow 0$. In fact, since the centred ball of radius $1$ can be regarded as a compact metric space, the family of probability measures $(\P_\nu)_{\nu>0}$ with densities $f_\nu(|x|):={\rm const}(1-|x|^2)^{(\nu-2)/2}$, $\nu>0$, is tight for each $n\in\N$. Thus, $(\P_\nu)_{\nu>0}$ is weakly sequentially compact, i.e., there exist weakly convergent subsequence $(\P_{\nu_n})_{n\in\N}$ with $\nu_n\downarrow 0$. For each such sequence $\nu_n$ the limiting probability measure is easily seen to have the following two properties: (i) it is rotation invariant and (ii) it is concentrated on the boundary of the centred ball of radius $1$, that is, the radius $1$ sphere. In other words, the limit must coincide with the normalized spherical Lebesgue measure on that sphere. Now, as $\nu\downarrow 0$ and since $(x_1,\ldots,x_{r+1})\mapsto \dist(0,\aff(x_1,\ldots,x_{r+1}))$ is a bounded continuous function on the $(r+1)$-st cartesian power of the unit ball, the density in (d) is the limit of the density in (b).
\end{proof}

\section{Cumulants, Berry-Esseen bounds and moderate deviations}

In this section we shall concentrate on the Gaussian, the Beta and the spherical model, for which the random variables $\cV_{n,r}$ have finite moments of all orders for any $n\in\N$ and $r\leq n$.

\subsection{Cumulants for logarithmic volumes}

For a random variable $X$ with $\E[|X|^m]<\infty$ for some $m\in\N$, we write $c^m[X]$ for the $m$th order cumulant of $X$, that is,
\begin{align}\label{DefinitionCumulant}
c^m[X] = (-\mathfrak{i})^{m}\,{\dd^m\over\dd  t^m}\log\bE[\exp(\mathfrak{i}tX)]\Big|_{t=0}\,,
\end{align}
where $\mathfrak{i}$ stands for the imaginary unit. It is well known that sharp bounds for cumulants lead to fine probabilistic estimates for the involved random variables. For the volume of a random simplex with Gaussian or Beta distributed vertices we shall establish the following cumulant bound. In what follows we shall write $a_n\sim b_n$ for two sequences $(a_n)_{n\in\N}$ and $(b_n)_{n\in\N}$ if $a_n/b_n\to 1$, as $n\to\infty$. Let us define the random variable $\cL_{n,r}:=\log(r!\cV_{n,r})$.

\begin{theorem}\label{thm:Cumulants}
Let $X_1,\ldots,X_{r+1}$ be chosen according to one of the four models presented in the previous section, and let $\alpha\in(0,1)$.
\begin{itemize}
\item[(a)] For the Gaussian model we have
$$
\E\cL_{n,r}\sim \frac{r}{2} \log{n},
\qquad\qquad
\Var\cL_{n,r} \sim \begin{cases}
{r\over 2n} &: r=o(n)\\
{\frac 12 \log \frac 1 {1-\alpha}}&: r\sim \alpha n\\
\frac{1}{2} \log \left(\frac{n}{n-r+1}\right) &: n-r = o(n)
\end{cases}
$$
and, for $m\geq 3$,
$$
|c^m(\cL_{n,r})|\leq \begin{cases}
 C^m (m-1)! rn^{1-m}  &: r=o(n) \text{ or } r\sim\alpha n\\
2\,(m-1)!&: \text{for arbitrary r(n)}\,,
\end{cases}
$$
where $C\in(0,\infty)$ is a constant not depending on $n$ and $m$.
\item[(b)] For the Beta model and the spherical model we have
\begin{align*}
\Var\cL_{n,r} &\sim 
\frac 12 \log \frac n{n-r} - \frac {r^2}{2n (r+1)}
\sim \begin{cases}
\frac {r}{2(r+1)n}&: r=o(\sqrt n)\\
\frac 12 \log \frac 1 {1-\alpha} -\frac \alpha 2&: r\sim \alpha n\\
\frac{1}{2} \log \left(\frac{n}{n-r+1}\right) &: n-r = o(n)
\end{cases}
\end{align*}
and, for all $m\geq 3$ and $n\geq 4$,
$$
|c^m(\cL_{n,r})|\leq \begin{cases}
C^m m!rn^{1-m} &: r=o(n)\text{ or }r\sim \alpha n\\
2 \cdot 4^m m! &: \text{for arbitrary r(n)}\,,
\end{cases}
$$
where $C\in(0,\infty)$ is a constant not depending on $n$ and $m$.
\end{itemize}
\end{theorem}

The proof of Theorem \ref{thm:Cumulants} is to some extent canonical and roughly follows \cite{DöringEichelsbacher}. In particular, it is based on an asymptotic analysis, as $|z|\to\infty$, of the digamma function $\psi(z)=\psi^{(0)}(z) : = \frac{\dd}{\dd z} \log\Gamma(z)$, and the polygamma functions
$$
\psi^{(m)}(z) : = \frac{\dd^m}{\dd z^m}\psi(z) = \frac{\dd^{m+1}}{\dd z^{m+1}} \log\Gamma(z),\qquad m \in \N.
$$
We start with the following lemma.

\begin{lemma}\label{lem:AsymptoticPolygamma}
Let $m\in\N$. Then, as $|z|\to\infty$ in $|\arg z|<\pi-\eps$,
\begin{equation}\label{eq:psi_asympt}
\psi(z) = \log z + O(1/z)\quad\text{and}\quad
\psi^{(m)}(z) =(-1)^{m-1}\,{(m-1)!\over z^m} + O(1/z^{m+1})\,.
\end{equation}
Moreover, for all $z>0$,
\begin{equation}\label{eq:psi_ineq}
|\psi^{(m)}(z)| \leq {(m-1)!\over z^m} + {m!\over z^{m+1}}\,.
\end{equation}
\end{lemma}
\begin{proof}
The asymptotic relations can be found in \cite{Abramovitz}, pp.\ 259--260. To prove the inequality, note that
$$
|\psi^{(m)}(z)| = \sum_{k=0}^\infty \frac { m!} {(z+k)^{m+1}} \leq \frac {m!}{z^{m+1}} + m!\int_z^\infty \frac {{\rm d}x}{x^{m+1}} = \frac {m!}{z^{m+1}} +  \frac{(m-1)!}{z^m},
$$
where we estimated the sum by the integral because the function $x\mapsto x^{-(m+1)}$, $x>0$, is decreasing.
\end{proof}

\begin{lemma}\label{lem:Polygamma(const)}
As $n\to\infty$, one has
\begin{equation}\label{eq:lem:Polygamma_asympt}
\frac{1}{2} \sum_{j=1}^{n} \psi\left(\frac{j}{2}\right) \sim  {n\over 2} \log n, 
\quad 
\frac{1}{4} \sum_{j=1}^{n} \psi^{(1)}\left(\frac{j}{2}\right) = {1\over 2}\log {n} +c_1  + o(1),
\end{equation}
where $c_1={1\over 2}(\gamma+1+{\pi^2\over 8})$ with the Euler-Mascheroni constant $\gamma$. Moreover, for all $m\ge 3$,
\begin{equation}\label{eq:lem:Polygamma_ineq}
\frac{1}{2^{m}} \Big| \sum_{j=1}^{n} \psi^{(m-1)}\left(\frac{j}{2}\right) \Big| \le 2 (m-1)!\,.
\end{equation}
\end{lemma}
\begin{proof}
The asymptotic relations~\eqref{eq:lem:Polygamma_asympt} can essentially be found in \cite{DöringEichelsbacher} (where the constant $c_1$ has been computed explicitly).  The first one follows from $\psi(z) = \log z + O(1/z)$ as $z\to\infty$ together with
$\sum_{j=1}^n \log \frac j2 \sim  n\log {n}$ as $n\to\infty$.
To prove the second one, write
$$
\frac 14\sum_{j=1}^{n} \psi^{(1)}\left(\frac{j}{2}\right) -\frac 12 \log n = \frac 14\sum_{j=1}^{n} \left(\psi^{(1)}\left(\frac{j}{2}\right)-\frac 2j \right) + \frac 12 \left(\sum_{j=1}^n \frac 1j - \log n\right)
$$
and observe that the series  $\sum_{j=1}^{\infty} (\psi^{(1)}(\frac{j}{2})-\frac 2j)$ converges because $\psi^{(1)}(z)- \frac 1z = O(\frac 1 {z^2})$ as $z\to\infty$. The claim follows since $\sum_{j=1}^n \frac 1j -\log n$ converges to the Euler constant $\gamma$.

To prove inequality~\eqref{eq:lem:Polygamma_ineq}, use Lemma \ref{lem:AsymptoticPolygamma} to get
$$
\frac{1}{2^{m}} \Big| \sum_{j=1}^{n} \psi^{(m-1)}\left(\frac{j}{2}\right) \Big| \leq \frac 1 {2^m} \sum_{j=1}^{\infty} \left(\frac{(m-2)!}{(j/2)^{m-1}} + \frac{(m-1)!}{(j/2)^m}\right)
\leq (m-1)! \left(\frac 14\zeta(2) + \zeta(3)\right)
$$
for all $m\geq 3$, where we used the inequality $(m-2)! \leq \frac 12 (m-1)!$.  The constant in the brackets is smaller than $2$.
\end{proof}

Since the moments of $\cV_{n,r}$ both, for the Gaussian and the Beta model, involve the same product of fractions of Gamma functions, we prepare the proof of Theorem \ref{thm:Cumulants} with the following lemma. We define
$$
S_{n,r}(z):=\sum_{j=1}^{r} \bigg[\log\Gamma\left({n-r+j+z\over 2}\right) - \log\Gamma\left({n-r+j\over 2}\right)\bigg],\qquad z>0.
$$

\begin{lemma}\label{lem:CumulantsPreparation}
\begin{itemize}
\item[(a)] If $r=o(n)$ then, as $n\to\infty$,
$$
{\dd^m\over\dd z^m}S_{n,r}(z)\Big|_{z=0}
\sim
\begin{cases}
{r\over 2} \log {n}  &: m=1 \\
{(-1)^{m}\over 2}\,(m-2)! \,r\, n^{-(m-1)} &: m\geq 2.
\end{cases}
$$
\item[(b)] If $r \sim \alpha n$ for some $\alpha\in(0,1)$ then, as $n\to\infty$,
$$
{\dd^m\over\dd z^m}S_{n,r}(z)\Big|_{z=0}
\sim
\begin{cases}
\frac {\alpha n}{2} \log n &: m=1\\
\frac 12 \log \frac 1{1-\alpha}&: m=2 \\
\frac{(-1)^{m}  (m-3)!}{2\cdot  n^{m-2}} \left(\frac 1 {(1-\alpha)^{m-2}} - 1\right)&: m\geq 3.
\end{cases}
$$
\item[(c)] If $n-r=o(n)$ then, as $n\to\infty$,
$$
{\dd^m\over\dd z^m}S_{n,r}(z)\Big|_{z=0} \sim   \begin{cases}
{n\over 2}\log {n}  &: m=1\\
{{1}\over 2}\log \frac {n}{n-r+1}&: m=2.
\end{cases}
$$
\item[(d)]
For $m\ge 2$ and if $r=o(n)$ or $r\sim \alpha n$, $\alpha\in (0,1)$, then there is a constant $C$ which may depend on $\alpha$ (but does not depend on $m,n$) such that
\begin{align*}
\left|{\dd^m\over\dd z^m}S_{n,r}(z)\Big|_{z=0}\right| \leq  C^m  (m-1)! r n^{1-m}.
\end{align*}
\item[(e)]
Finally, for $m\ge 3$ and without any conditions on $r$, we have
\begin{align*}
\left|{\dd^m\over\dd z^m}S_{n,r}(z)\Big|_{z=0}\right| \le 2(m-1)!\,.
\end{align*}
\end{itemize}

\end{lemma}
\begin{proof}
Let us prove $(a)$, $(b)$, $(c)$ for $m=1$. We have
$$
{\dd\over\dd z}S_{n,r}(z)\Big|_{z=0} = {1\over 2}\sum_{j=1}^r\psi\left(\frac{n-r+j}{2}\right)
={1\over 2}\sum_{j=1}^{n}\psi\left(\frac{j}{2}\right) - {1\over 2}\sum_{j=1}^{n-r}\psi\left(\frac{j}{2}\right),
$$
and all three statements follow easily from the relation $\frac{1}{2} \sum_{j=1}^{n} \psi\left(\frac{j}{2}\right) \sim  {n\over 2} \log n$; see Lemma \ref{lem:Polygamma(const)}.

Next we prove $(a)$, $(b)$, $(c)$ for $m\geq 2$. We have
$$
{\dd^m\over\dd z^m}S_{n,r}(z)\Big|_{z=0} = \frac{1}{2^m} \sum_{j=1}^{r} \psi^{(m-1)}\left(\frac{n-r+j}{2}\right)
$$
and again we can conclude $(a)$ by using Equation~\eqref{eq:psi_asympt} of Lemma \ref{lem:AsymptoticPolygamma}. To prove $(b)$  for $m=2$, apply the second asymptotics in~\eqref{eq:lem:Polygamma_asympt} of Lemma \ref{lem:Polygamma(const)} to get
\begin{align*}
{\dd^2\over\dd z^2}S_{n,r}(z)\Big|_{z=0} &= \frac{1}{4} \sum_{j=1}^{r} \psi^{(1)}\left(\frac{n-r+j}{2}\right) = \frac 12 \log n +c_1  - \frac 12 \log (n-r)-c_1 + o(1)\\
&= \frac 12 \log \frac n {n-r} + o(1) = \frac 12 \log \frac 1 {1-\alpha} + o(1).
\end{align*}
To prove $(b)$ for $m\geq 3$, note that for $r\sim \alpha n$,
\begin{align*}
&\frac{1}{2^m} \sum_{j=1}^{r} \psi^{(m-1)}\left(\frac{n-r+j}{2}\right) \sim
\frac{1}{2^m} \sum_{k=n-r+1}^n\,\frac {(-1)^{m-2} (m-2)!} {(k/2)^{m-1}}\\
&=
\frac{(-1)^m\,(m-2)!}{2} \bigg[\sum_{k=1}^n\,\frac {1} {k^{m-1}}-\sum_{k=1}^{n-r}\,\frac {1} {k^{m-1}}\bigg]\sim
\frac{(-1)^{m}  (m-3)!}{2\cdot  n^{m-2}} \left(\frac 1 {(1-\alpha)^{m-2}} - 1\right),
\end{align*}
using the asymptotics for the tail of the Riemann zeta series. Finally, to prove $(c)$ for $m=2$ use the formula $\frac{1}{4} \sum_{j=1}^{r} \psi^{(1)}\left(\frac{n-r+j}{2}\right) = \frac 12 \log n + O(1)$ following from~\eqref{eq:lem:Polygamma_asympt} to get
\begin{multline*}
{\dd^2\over\dd z^2}S_{n,r}(z)\Big|_{z=0} = \frac{1}{4} \sum_{j=1}^{r} \psi^{(1)}\left(\frac{n-r+j}{2}\right) = \frac 12 \log n +O(1)  - \frac 12 \log (n-r+1) - O(1)\\
= \frac 12 \log \frac n {n-r+1} + O(1) \sim \frac 12 \log \frac n {n-r+1}
\end{multline*}
because $\frac n {n-r+1}\to\infty$. We added the term $+1$ to make the formula work in the case $r=n$. 


Let us prove $(d)$. Since the function $|\psi^{(m-1)}(z)| = \sum_{k=0}^\infty \frac {(m-2)!} {(z+k)^{m}}$ is decreasing, we can write
$$
\left| {\dd^m\over\dd z^m}S_{n,r}(z)\Big|_{z=0}\right|
=
\frac{1}{2^m} \sum_{j=1}^{r} \left|\psi^{(m-1)}\left(\frac{n-r+j}{2}\right)\right|
\leq
\frac {r}{2^m} \left|\psi^{(m-1)}\left(\frac{n-r+1}{2}\right)\right|,
$$
and the claim follows from the estimates $|\psi^{(m-1)}(z)| \leq 2 \cdot (m-1)! z^{1-m}$, $z\geq 1$, which is a consequence of Lemma \ref{lem:AsymptoticPolygamma}, and $n-r+1 > n/C$ for sufficiently large $C$.

Let us prove $(e)$.  If $m\ge 3$ and $r$ is arbitrary, we observe that the function $\psi^{(m-1)}(z)$, $z>0$, has the same sign as $(-1)^m$ and hence
\begin{align*}
\left| {\dd^m\over\dd z^m}S_{n,r}(z)\Big|_{z=0}\right|
&=
\frac{1}{2^m} \sum_{j=1}^{r} \left|\psi^{(m-1)}\left(\frac{n-r+j}{2}\right)\right|
\leq  \frac{1}{2^m} \sum_{j=1}^{n} \left|\psi^{(m-1)}\left(\frac{j}{2}\right)\right|.
\end{align*}
Then, the result follows in view of inequality~\eqref{eq:lem:Polygamma_ineq} of Lemma \ref{lem:Polygamma(const)}. Thus, the proof is complete.
\end{proof}

\begin{proof}[Proof of Theorem \ref{thm:Cumulants}]
Denote the moment generating function of $\cL_{n,r}=\log (r! \cV_{n,r})$ by
$$
M_{n,r}(z):=\bE[\exp(z\cL_{n,r})]= \E[(r!\cV_{n,r})^z].
$$
We start with the Gaussian model. Recalling the moment formula from Theorem \ref{theo:vol_Simplices_Moments}(a), we see that
\begin{align*}
\log M_{n,r}(z) = S_{n,r}(z)+\frac{z}{2} \log(r+1) + \frac{zr}{2} \log 2
\end{align*}
and hence
\begin{align*}
\frac{\dd^m}{\dd z^m} \log M_{n,r}(z) = \frac{\dd^m}{\dd z^m}S_{n,r}(z)+{\bf 1}_{\{m=1\}}\frac{1}{2} \log(r+1) +  {\bf 1}_{\{m=1\}} \frac{r}{2} \log 2
\end{align*}
for all $m\in\N$. By taking $z=0$ it follows that
$$
c^m[\cL_{n,r}] = \frac{\dd^m}{\dd z^m}S_{n,r}(z)\Big|_{z=0} +{\bf 1}_{\{m=1\}}\frac{1}{2} \log(r+1) +  {\bf 1}_{\{m=1\}} \frac{r}{2} \log 2.
$$
Using Lemma~\ref{lem:CumulantsPreparation} we immediately get the required asymptotic formulae for $\E \cL_{n,r} = c^1[\cL_{n,r}]$ and $\Var \cL_{n,r} = c^2[\cL_{n,r}]$. The estimates for the cumulants $c^m[\cL_{n,r}]$, $m\geq 3$,  follow from Lemma~\ref{lem:CumulantsPreparation} (d),(e).

\vspace*{2mm}
Next, we consider the Beta model and prove part (b) of the theorem. Recalling the moment formula from Theorem \ref{theo:vol_Simplices_Moments}(b) and denoting by $M_{n,r}(z)$ again the moment generating function of $\cL_{n,r}$, we see that
\begin{multline*}
\log M_{n,r}(z) =
S_{n,r}(z)+\log\Gamma\Big({r(n+\nu-2)+(n+\nu)\over 2}+\frac{(r+1)z}{2}\Big)
\\
+(r+1)\log \Gamma\left(\frac{n + \nu}{2}\right)
-\log \Gamma\Big({r(n+\nu-2)+(n+\nu)\over 2}+\frac{rz}{2}\Big)-(r+1)\log \Gamma\left(\frac{n+\nu}{2} + \frac{z}{2}\right).
\end{multline*}
It follows that, for $m\in\N$, $\frac{\dd^m}{\dd z^m} \log M_{n,r}(z)$ equals
\begin{equation}\label{eq:AbleitungMnrZ}
\begin{split}
&\frac{\dd^m}{\dd z^m}S_{n,r}(z)+\Big({r+1\over 2}\Big)^{m}\psi^{(m-1)}\Big({r(n+\nu-2)+(n+\nu)\over 2}+\frac{(r+1)z}{2}\Big)\\
&-\Big({r\over 2}\Big)^{m}\psi^{(m-1)}\Big({r(n+\nu-2)+(n+\nu)\over 2}+\frac{rz}{2}\Big)-{r+1\over 2^{m}}\psi^{(m-1)}\left(\frac{n+\nu}{2} + \frac{z}{2}\right).
\end{split}
\end{equation}
Taking $z=0$, we obtain
\begin{multline}\label{eq:c_m_expression}
c^m[\cL_{n,r}] =
\frac{\dd^m}{\dd z^m} S_{n,r}(z)\Big|_{z=0}  + \Big({r+1\over 2}\Big)^{m}\psi^{(m-1)}\Big({r(n+\nu-2)+(n+\nu)\over 2}\Big)\\
-\Big({r\over 2}\Big)^{m}\psi^{(m-1)}\Big({r(n+\nu-2)+(n+\nu)\over 2}\Big)-{r+1\over 2^{m}}\psi^{(m-1)}\left(\frac{n+\nu}{2}\right).
\end{multline}

Let us compute the asymptotics of $\Var\cL_{n,r}= c^2[\cL_{n,r}]$ in the case $r=o(n)$.  First of all, using the formula $\psi^{(1)}(z) = 1/z + O(1/z^2)$ as $z\to\infty$, we obtain
$$
{\dd^2\over\dd z^2}S_{n,r}(z)\Big|_{z=0} = \frac{1}{4} \sum_{j=1}^{r} \psi^{(1)}\left(\frac{n-r+j}{2}\right) 
=
\frac 14 \sum_{j=1}^r \frac {2}{n-r+j} + O\left(\frac {r}{n^2}\right)
=
\frac {H_n-H_{n-r}}2 + O\left(\frac {r}{n^2}\right),
$$
where $H_n = \sum_{k=1}^n 1/k$ is the $n$-th harmonic number. Using the formula $H_n = \log n + \gamma + 1/(2n) + O(1/n^2)$ as $n\to\infty$, we arrive at
$$
{\dd^2\over\dd z^2}S_{n,r}(z)\Big|_{z=0}
=
\frac 12 \log \frac{n}{n-r} + \frac 12 \left(\frac 1n - \frac 1 {n-r}\right) + O\left(\frac 1 {n^2}\right) +  O\left(\frac {r}{n^2}\right)
=
\frac 12 \log \frac{n}{n-r} + O\left(\frac {r}{n^2}\right). 
$$
Again using the formula $\psi^{(1)}(z) = 1/z + O(1/z^2)$ as $z\to\infty$, we obtain
$$
\psi^{(1)}\Big({r(n+\nu-2)+(n+\nu)\over 2}\Big) = \frac 2 {n(r+1) + O(r)} + O\left(\frac {1}{n^2r^2}\right) =\frac 2 {n(r+1)} + O\left(\frac {1}{n^2r}\right)
$$
and
$$
\psi^{(1)}\left(\frac{n+\nu}{2}\right) = \frac {2}{n} + O\left(\frac 1 {n^2}\right).
$$
Recalling~\eqref{eq:c_m_expression} and taking everything together, we obtain
\begin{multline*}
\Var\cL_{n,r} = c^2[\cL_{n,r}]
=
\frac 12 \log \frac{n}{n-r} + \frac {2r+1}{4}\frac 2 {n(r+1)} - \frac {r+1}{4}\frac {2}{n}  + O\left(\frac r {n^2}\right)
\\=
\frac 12 \log \frac n{n-r} - \frac {r^2}{2n (r+1)}+ O\left(\frac r {n^2}\right).
\end{multline*}

In the case $r\sim \alpha n$ we evidently have 
$$
\lim_{n\to\infty} \Var\cL_{n,r} = \frac 12 \log \frac 1 {1-\alpha} - \frac {\alpha}{2}.
$$
In the case $r= o(n)$ observe that  $\log \frac{n}{n-r} \geq \frac rn$, so that $\frac 12 \log \frac n{n-r} - \frac {r^2}{2n (r+1)} \geq \frac {r}{2n (r+1)}$. Thus, we have $\frac r {n^2}=o(\frac 12 \log \frac n{n-r} - \frac {r^2}{2n (r+1)})$ and we can conclude that
$$
\Var\cL_{n,r} \sim \frac 12 \log \frac n{n-r} - \frac {r^2}{2n (r+1)}. 
$$
Finally, observe that in the case $r=o(\sqrt n)$ we can use the Taylor expansion of the logarithm to get $\Var\cL_{n,r} \sim \frac {r}{2(r+1)n}$, but this formula breaks down if $r$ is of order $\sqrt n$. This completes the proof of the asymptotics of $\Var\cL_{n,r}$ in the cases $r=o(n)$ and $r\sim \alpha n$.

\vspace*{2mm}

Let us now compute the asymptotics of $\Var\cL_{n,r})= c^2[\cL_{n,r}]$ in the case $n-r = o(n)$.  Using the formula $\psi^{(1)}(z) = 1/z + O(1/z^2)$ as $z\to\infty$, we obtain
$$
{\dd^2\over\dd z^2}S_{n,r}(z)\Big|_{z=0} = \frac{1}{4} \sum_{j=1}^{r} \psi^{(1)}\left(\frac{n-r+j}{2}\right) 
=
\frac {H_n-H_{n-r}}2 + O\left(\frac 1n\right),
$$
Using the formulas $H_n = \log n + O(1)$ and $H_{n-r} = \log (n-r+1) +O(1)$ (where  $+1$ is needed to make the expression well-defined in the case $r=n$), we arrive at
$$
{\dd^2\over\dd z^2}S_{n,r}(z)\Big|_{z=0} = \frac 12 \log \frac{n}{n-r+1} +O(1).
$$
By the formula $\psi^{(1)}(z) = O(1/z)$ as $z\to\infty$, we have
$$
\psi^{(1)}\Big({r(n+\nu-2)+(n+\nu)\over 2}\Big) = O\left(\frac 1{n^2}\right),
\quad
\psi^{(1)}\left(\frac{n+\nu}{2}\right) =  O\left(\frac 1 {n}\right).
$$
Plugging everything into \eqref{eq:c_m_expression} yields
$$
\Var\cL_{n,r} = c^2[\cL_{n,r}] = \frac 12 \log \frac{n}{n-r+1} +O(1)
\sim \frac 12 \log \frac{n}{n-r+1}
$$
because $\frac {n}{n-r+1} \to\infty$, thus proving the required asymptotics of the variance.

\vspace*{2mm}
Next we prove the bounds on the cumulants assuming that $r=o(n)$ or $r\sim \alpha n$. Recall from Lemma \ref{lem:CumulantsPreparation}(d) the estimate
\begin{align*}
\left|{\dd^m\over\dd z^m}S_{n,r}(z)\Big|_{z=0}\right| \leq  C^m  (m-1)! r n^{1-m}.
\end{align*}
Further, since $\nu\geq 0$, we have
$$
{r(n+\nu-2)+(n+\nu)\over 2} \geq \frac {r(n-2)}{2}. 
$$
Since the function $|\psi^{(m-1)}(z)|$ is non-increasing, we have, using also the estimate $|\psi^{(m-1)}(z)| \leq 2 \cdot (m-1)! z^{1-m}$,
$$
\left|\psi^{(m-1)}\Big({r(n+\nu-2)+(n+\nu)\over 2}\Big)\right| \leq \left|\psi^{(m-1)}\Big(\frac{r(n-2)}{2}\Big) \right|
\leq 
2^m (m-1)! r^{1-m} (n-2)^{1-m}.
$$
By the mean value theorem, we also have $(r+1)^m - r^m \leq m (r+1)^{m-1}$, hence
$$
\frac{(r+1)^m - r^m}{2^m}\left|\psi^{(m-1)}\Big({r(n+\nu-2)+(n+\nu)\over 2}\Big)\right|
\leq 
m! \left(\frac{r+1}{r}\right)^{m-1}(n-2)^{1-m} 
\leq 
4^m m!  n^{1-m}
$$
because $n-2\geq n/2$ for $n\geq 4$. 
Similarly, by the non-increasing property of $|\psi^{(m-1)}(z)|$ and the estimate $|\psi^{(m-1)}(z)| \leq 2 \cdot (m-1)! z^{1-m}$, we have
$$
{r+1\over 2^{m}}\left|\psi^{(m-1)}\left(\frac{n+\nu}{2}\right) \right| 
\leq {r+1\over 2^{m}} \left|\psi^{(m-1)}\left(\frac{n}{2}\right)\right| 
\leq 2r (m-1)! n^{1-m}.  
$$
Recalling \eqref{eq:c_m_expression} and taking the above estimates together, we arrive at the required estimate
$$
|c^m[\cL_{n,r}]| \leq  C^m  m! r n^{1-m}
$$
for a sufficiently large constant $C>0$ not depending on $n$ and $m$.
To prove the bound $|c^m[\cL_{n,r}]| \leq   2 \cdot 4^m m!$ without restrictions on $r(n)$, we argue as above except that using Lemma~\eqref{lem:CumulantsPreparation}(e) to bound the derivative of $S_{n,r}$:
$$
|c^m[\cL_{n,r}]| \leq 2 (m-1)! + 4^m m!  n^{1-m} + 2r (m-1)! n^{1-m} \leq 2 \cdot 4^m m!. 
$$

Finally, we consider the spherical model. Since the results for the Beta model are independent of the parameter $\nu$, they carry over to the spherical model which appears as a limiting case, as $\nu\downarrow 0$.
\end{proof}

\subsection{Berry-Esseen bounds and moderate deviations for the log-volume}\label{SectionLDP}

We introduce some terminology. One says that a sequence $(\nu_n)_{n \in \bN}$ of probability measures on a topological space $E$ fulfils a large deviation principle (LDP) with speed $a_n$ and (good) rate function $I : E \rightarrow [0,\infty]$, if $I$ is lower semi-continuous, has compact level sets and if for every Borel set $B\subseteq E$,
\begin{align*}
-\inf\limits_{x\in \intt(B)} I(x) \leq \liminf\limits_{n \rightarrow \infty} a_n^{-1} \log \nu_n (B) \leq \limsup\limits_{n \rightarrow \infty} a_n^{-1} \log \nu_n (B) \leq -\inf\limits_{x\in \cl(B)} I(x)\,,
\end{align*}
where $\intt(B)$ and $\cl(B)$ stand for the interior and the closure of $B$, respectively. A sequence $(X_n)_{n \in \bN}$ of random elements in $E$ satisfies a LDP with speed $a_n$ and rate function $I : E \rightarrow [0,\infty]$, if the family of their distributions does. Moreover, if the rescaling $a_n$ lies between that of a law of large numbers and that of a central limit theorem, one usually speaks about a moderate deviation principle (MDP) instead of a LDP with speed $a_n$ and rate function $I$.

We shall say that a sequence of real-valued random variables $(X_n)_{n\in\N}$ satisfying $\E|X_n|^2<\infty$ for all $n\in\N$ fulfils a Berry-Esseen bound (BEB) with speed $(\varepsilon_n)_{n\in\N}$ if
$$
\sup_{t\in\R}\Big|\P\Big({X_n-\E[X_n]\over\sqrt{\Var X_n}}\leq t\Big)-\Phi(t)\Big| \leq c\,\varepsilon_n\,,
$$
where $c>0$ is a constant not depending on $n$ and $\Phi(\,\cdot\,)$ denotes the distribution function of a standard Gaussian random variable.

\begin{theorem}[BEB and MDP for the log-volume]\label{thm:CLTMDPSimplices}
Let $X_1,\ldots,X_{r+1}$ be chosen according to the Gaussian, the Beta or the spherical model.
\begin{itemize}
\item[(a)] For the Gaussian model define
$$
\varepsilon_n := {1\over\sqrt{rn}}\text{ if $r=o(n)$ or $r\sim\alpha n$}\quad\text{and}\quad \varepsilon_n:={1\over \sqrt{\log \left(\frac{n}{n-r+1}\right)}}\text{ if $n - r = o(n)$},
$$
where $\alpha\in(0,1)$.
Then $\cL_{n,r}$ satisfies a BEB with speed $\varepsilon_n$. Further, let $(a_n)_{n\in\N}$ be such that $a_n\to\infty$ and $a_n\varepsilon_n\to 0$, as $n\to\infty$. Then $\cL_{n,r}$ satisfies a MDP with speed $a_n$ rate function $I(x)={x^2\over 2}$.
\item[(a)] For the Beta model and the spherical model define
$$
\varepsilon_n := {1\over\sqrt{n}}\text{ if $r=o(n)$},\quad \varepsilon_n:={1\over n}\text{ if $r\sim\alpha n$}\quad\text{and}\quad \varepsilon_n:={1\over \sqrt{\log \left(\frac{n}{n-r+1}\right)}}\text{ if $n - r = o(n)$}
$$
with $\alpha\in(0,1)$.
Then $\cL_{n,r}$ satisfies a BEB with speed $\varepsilon_n$. Further, let $(a_n)_{n\in\N}$ be such that $a_n\to\infty$ and $a_n\varepsilon_n\to 0$, as $n\to\infty$. Then $\cL_{n,r}$ satisfies a MDP with speed $a_n$ rate function $I(x)={x^2\over 2}$.
\end{itemize}
\end{theorem}
\begin{proof}
Let us define the normalized random variable $\widetilde{\cL}_{n,r}:=(\cL_{n,r}-\E[\cL_{n,r}])/\sqrt{\Var\cL_{n,r}}$. From Theorem \ref{thm:Cumulants} we conclude that, for $m\geq 3$,
$$
|c^m(\widetilde{\cL}_{n,r})|
=
\frac{|c^m(\cL_{n,r})|}{(\Var \cL_{n,r})^{m/2}}
\leq \begin{cases}
\frac{c_1^m (m-1)!}{(\sqrt {rn})^{m-2}}   &: r=o(n) \text{ or } r\sim\alpha n\\
\frac{c_2^m (m-1)!}{\left(\sqrt{\log \frac{n}{n-r+1}}\right)^{m}} &: n - r = o(n)
\end{cases}
$$
in the Gaussian case and
$$
|c^m(\widetilde{\cL}_{n,r})| \leq \begin{cases}
{|c^m[\cL_{n,r}]|\over(r/(2(r+1)n))^{m/2}} &: r=o(n)\\
{|c^m[\cL_{n,r}]|\over(\frac{1}{2}\log(\frac{1}{1-\alpha}) - \frac{\alpha}{2})^{m/2}} &: r\sim\alpha n\\
{|c^m[\cL_{n,r}]|\over(\frac{1}{2} \log \left(\frac{n}{n-r+1}\right) )^{m/2}} &: n - r = o(n)
\end{cases}
\leq \begin{cases}
c_4^m\,m!\, \big({1\over\sqrt{n}}\big)^{m-1} &: r=o(n)\\
c_5^m\,m!\, \big({1\over n}\big)^{m-1} &: r\sim\alpha n\\
c_6^m\,m!\, \left(\frac{1}{\sqrt{\log \left(\frac{n}{n-r+1}\right)}}\right)^{m - 2} &: n - r = o(n)
\end{cases}
$$
for the Beta and the spherical model with constants $c_1,\ldots,c_6>0$ not depending on $m$ and $n$. The result follows now from \cite[Theorem 1.1]{DoeringEichelsbacher} and \cite[Corollary 2.1]{SaulisBuch}.
\end{proof}

\begin{remark}
Starting with the cumulant bounds presented in Theorem \ref{thm:Cumulants} one can also derive
\begin{itemize}
\item[(i)] concentration inequalities,
\item[(ii)] bound for moments of all orders,
\item[(iii)] Cram\'er-Petrov type results concerning the relative error in the central limit theorem,
\item[(iv)] strong laws of large numbers
\end{itemize}
for the random variables $\widetilde{\cL}_{n,r}$ from the results presented in \cite[Chapter 2]{SaulisBuch} (see also \cite{GroteThäle,GroteThäle2}).
\end{remark}

While in the three cases $r=o(n)$, $r\sim\alpha n$ and $n-r=o(n)$ we were able to derive precise Berry-Esseen bounds by using cumulant bounds, we can state a `pure' central limit theorem for the log-volume in an even more general setup. The following result can directly be concluded by extracting subsequence and then by applying the result of Theorem \ref{thm:CLTMDPSimplices}.

\begin{corollary}[Central limit theorem for the log-volume]\label{cor:log_volume_CLT}
Let $r=r(n)$ be an arbitrary sequence of integers such that $r(n)\leq n$ for any $n\in\N$. Further, let for each $n\in\N$, $X_1,\ldots,X_{r+1}$ be independent random points chosen according to the Gaussian, the Beta or the spherical model, and put $\cL_{n,r}:=\log(r!\cV_{n,r})$. Then,
\begin{align*}
\frac{\cL_{n,r} - \bE[\cL_{n,r}]}{\sqrt{\Var\cL_{n,r}}} \todistr Z,
\end{align*}
where $Z\sim\cN(0,1)$ is a standard Gaussian random variable.
\end{corollary}

\subsection{Central and non-central limit theorem for the volume}

After having investigated asymptotic normality for the log-volume of a random simplex, we turn now to its actual volume, that is, the random variable $\cV_{n,r}$.

\begin{theorem}[Distributional limit theorem for the volume]\label{Volume}
Let $X_1,\ldots,X_{r+1}$ be chosen according to the Gaussian model, the Beta model or the spherical model and let $\alpha\in(0,1)$. Let $Z\sim\cN(0,1)$ be a standard Gaussian random variable.
\begin{enumerate}
\item If $r=o(n)$, then for suitable normalizing sequences $a_{n,r}$ and $b_{n,r}$ the following convergence in distribution holds, as $n\to\infty$:
$$
\frac{\cV_{n,r} - a_{n,r}}{b_{n,r}}\todistr Z.
$$
\item If $r\sim \alpha n$ for some $\alpha\in (0,1)$, then for a suitable normalizing sequence $b_{n,r}$ we have
$$
\frac{\cV_{n,r}}{b_{n,r}} \todistr
\begin{cases}
e^{\sqrt{\frac 12 \log \frac 1 {1-\alpha}}\, Z} &: \text{in the Gaussian model}\\
e^{\sqrt{\frac 12 \log \frac 1 {1-\alpha} - \frac \alpha 2}\, Z} &: \text{in the Beta or spherical model}.\\
\end{cases}
$$
\end{enumerate}
\end{theorem}

\begin{remark}
In the third case, i.e., if $n-r=o(n)$, there is no non-trivial distributional limit theorem for the random variable $\cV_{n,r}$ under affine re-scaling. The reason is that the variance of $\log \cV_{n,r}$ tends to $+\infty$ in this case. 
\end{remark}

The main ingredient in the proof of Theorem~\ref{Volume} in the case where $r=o(n)$ is the so-called 'Delta-Method', which is well known and commonly used in statistics, cf.\ \cite[Lemma 5.3.3]{BickelDoksum}.

\begin{proof}[Proof of Theorem \ref{Volume}]
From Corollary~\ref{cor:log_volume_CLT} we know that with the sequences $c_{n,r}= \E \log \cV_{n,r}$ and $d_{n,r}= \sqrt{\Var \log \cV_{n,r}}$ it holds that
$$
\frac{\log \cV_{n,r} -  c_{n,r}}{d_{n,r}} \todistr Z.
$$
By the Skorokhod--Dudley lemma~\cite[Theorem 4.30]{Kallenberg}, we can construct random variables $\cV_{n,r}^*$ and $Z^*$ on a different probability space such that $\cV_{n,r}^* \overset{d}{=} \cV_{n,r}$, $Z^* \overset{d}{=} Z$, and
$$
Z_n^* := \frac{\log \cV_{n,r}^* - c_{n,r}}{d_{n,r}} \toas Z^*.
$$
So, we have $\cV_{n,r}^* = e^{d_{n,r} Z_n^* + c_{n,r}}$, where $Z_n^*\to Z^*$ a.s., as $n\to\infty$.

Consider first the Gaussian model in the case $r\sim \alpha n$. Then, by Theorem~\ref{thm:Cumulants}(a)  we have
$$
d_{n,r} =\sqrt{\Var \log \cV_{n,r}} \sim \sqrt{\frac 12 \log \frac 1 {1-\alpha}}.
$$
With the aid of Slutsky's lemma it follows that
$$
\frac{\cV_{n,r}^*}{e^{c_{n,r}}} = e^{d_{n,r} Z_n^*} \toas e^{\sqrt{\frac 12 \log \frac 1 {1-\alpha}}\, Z^*}.
$$
Passing back to the original probability space, we obtain the distributional convergence
$$
\frac{\cV_{n,r}}{e^{c_{n,r}}}\todistr e^{\sqrt{\frac 12 \log \frac 1 {1-\alpha}}\, Z}.
$$
The proof for the Beta or spherical model in the case $r\sim \alpha n$ is similar, only the expression for the asymptotic variance being different.

Consider now the Gaussian model in the case $r=o(n)$. Then, by Theorem~\ref{thm:Cumulants}(a),
$$
d_{n,r} =\sqrt{\Var \log \cV_{n,r}} \ton 0.
$$
Using the formula $\lim_{x\to 0} (e^x-1)/x = 1$ and the Slutsky lemma, we obtain
$$
\frac{\frac{\cV_{n,r}^*}{e^{c_{n,r}}} - 1}{d_{n,r}}  = \frac{ e^{d_{n,r} Z_n^*} -1}{d_{n,r}Z_n^*} \cdot Z_n^* \toas Z^*.
$$
Passing back to the original probability space and taking $b_{n,r} = e^{c_{n,r}}d_{n,r}$ and $a_{n, r} =e^{c_{n,r}}$, we obtain the required distributional convergence.
\end{proof}

\section{Mod-$\phi$ convergence}\label{sec:mod_phi}
\subsection{Definition}
Mod-$\phi$ convergence is a powerful notion that was introduced and studied in \cite{jacod_etal,kowalski_nikeghbali,delbaen_etal,kowalski_najnudel_etal,Feray1}, to mention only some references. Once an appropriate version of  mod-$\phi$ convergence has been established, one gets for free a whole collection of limit theorems including the central limit theorem, the local limit theorem, moderate and large deviations, and a Cram\'er--Petrov asymptotic expansion~\cite{Feray1}.

The aim of the present Section~\ref{sec:mod_phi} is to establish mod-$\phi$ convergence for the log-volumes of the random simplices. Note that the mod-$\phi$ convergence we establish in the present
section together with the general results from \cite{Feray1} also
imply some of the results we proved in the previous section by means
of the cumulant method. On the other hand, we would like to emphasize
that this is not the case if $r\sim \alpha n$, for example.\\
 There are many definitions of mod-$\phi$ convergence. Here, we use one of the strongest ones, c.f.~\cite[Definition 1.1]{Feray1}.
Consider a sequence of random variables $(X_n)_{n\in\N}$ with moment generating functions $\varphi_n(t) = \E [e^{tX_n}]$ defined on some strip $S= \{z\in \mathbb C: c_- < \text{Re}\, t < c_+\}$. The sequence $(X_n)_{n\in\N}$ \textit{converges in the mod-$\phi$ sense}, where $\phi$ is an infinite-divisible distribution with moment generating function $\int_{-\infty}^{\infty} e^{tx}\phi(dx) = e^{\eta(t)}$, if
$$
\lim_{n\to\infty} \frac{\E [e^{tX_n}]}{e^{w_n \eta(t)}} = \psi(t)
$$
locally uniformly on $S$, where $(w_n)_{n\in\N}$ is some sequence converging to $+\infty$, and $\psi(t)$ is an analytic function on $S$.  As explained in references cited above,  mod-$\phi$ convergence roughly means that $X_n$ has approximately the same distribution as the $w_n$-th convolution power of the infinitely divisible distribution $\phi$. The ``difference'' between these distributions is measured by the ``\textit{limit function}'' $\psi$ that plays a crucial r\^{o}le in the theory.

\subsection{The Barnes $G$-function}
The Barnes function is an entire function of the complex argument $z$ defined by the Weierstrass product
$$
G(z) = (2\pi)^{z/2} \eee^{-\frac 12 (z+ (1+\gamma) z^2)} \prod_{k=1}^\infty \left(1+\frac zk\right)^k \eee^{\frac {z^2}{2k} - z},
$$
where $\gamma$ is the Euler constant.
The Barnes $G$-function satisfies the functional equation
$$
G(z+1) = \Gamma(z) G(z).
$$
By induction, one deduces that for all $n\in\N_0$,
\begin{equation}\label{eq:prod_gamma_functions}
\prod_{k=1}^n \Gamma(k+z) = \frac{G(z+n+1)}{G(z+1)}.
\end{equation}
We shall need the Stirling-type formula for $G$, see~\cite[p.~285]{barnes},
\begin{equation}\label{eq:stirling_for_G}
\log G(z+1) = \frac 12 z^2 \log z - \frac 34 z^2 + \frac z2 \log (2\pi) - \frac 1 {12} \log z + \zeta'(-1) + O(1/z),
\end{equation}
uniformly as $|z|\to +\infty$ such that $|\arg z| < \pi-\eps$,  where $\zeta'(-1)$ is the derivative of the Riemann $\zeta$-function. The value of $\zeta'(-1)$ can be expressed through the Glaisher--Kinkelin constant, but it cancels in all our calculations because we use~\eqref{eq:stirling_for_G} only via the following lemma.

\begin{lemma}\label{lem:barnes_G_diff_asympt}
Let  $|z|\to\infty$ such that $|\arg z| < \pi-\eps$. Let also $a=a(z)\in \C$ be such that $a/z\to 0$. Then, we have
$$
\log G(z+a+1) - \log G(z+1) = a \left(z\log z - z + \log \sqrt{2\pi}\right) + \frac 12 a^2 \log z +  O\left(\frac {|a|^3+1}{z}\right).
$$
\end{lemma}
\begin{proof}
Applying~\eqref{eq:stirling_for_G} we obtain that
$$
\log G(z+a+1) - \log G(z+1) = \frac 12 A_n + B_n + C_n + D_n + O(1/z),
$$
where
\begin{align*}
A_n
&=
(z+a)^2 \log (z+a) -z^2 \log z\\
&=
(z^2+a^2 + 2za) \left(\log z + \frac az - \frac {a^2}{2z^2} + O\left(\frac {a^3}{z^3}\right)\right) -z^2 \log z\\
&=
za - \frac 12 a^2 + a \log z + 2za \log z + 2a^2 + O\left(\frac {a^3}{z}\right),\\
B_n &= -\frac 34 \left((z+a)^2 - z^2 \right) = -\frac 34 a^2 - \frac 32 za, \\
C_n &= \frac 12 a \log (2\pi),\\
D_n &= -\frac 1 {12} (\log(z+a) - \log z) = O\left(\frac az\right).
\end{align*}
Taking everything together we get
$$
\log G(z+a+1) - \log G(z+1) =  a \left(z\log z - z + \log \sqrt{2\pi}\right) + \frac 12 a^2 \log z + O\left(\frac {|a|^3+1}{z}\right)
$$
and complete the proof of the lemma.
\end{proof}

\subsection{Mod-$\phi$ convergence for fixed $r\in \N$}

Recall that $\mathcal V_{n,r}$ denotes the volume of an $r$-dimensional random simplex in $\R^n$ whose $r+1$ vertices are distributed according to one of the four models presented in Section $1$.
We define, as usual, $\cL_{n,r} := \log (r! \cV_{n,r})$.
The next two propositions show that if $r\in\N$ is fixed, we have mod-$\phi$ convergence.

\begin{proposition}\label{prop:mod_phi_constant_r}
Fix some $r\in\N$ and consider the Gaussian model.   Then, as $n\to\infty$, the sequence  $n(\cL_{n,r}- \frac r2 \log n-{1\over 2}\log(r+1))$ converges in the mod-$\phi$ sense with $\eta(t) =\frac{1}{2}( (t+1)\log (t+1) - t)$  and parameter $w_n = rn$, namely
$$
\lim_{n\to\infty} \frac{\E \eee^{tn (\cL_{n,r}- \frac r2 \log n-{1\over 2}\log(r+1))}}{\eee^{rn \eta(t)}} = (t+1)^{-\frac {r(r+1)}{4} }
$$
uniformly as long as $t$ stays in any compact subset of $\C\setminus(-\infty,-1)$.
\end{proposition}
\begin{proof} 
An important formula we will often use describes the asymptotic behaviour of the Gamma function; it can be found in \cite[Eq.~6.1.39 on p.~257]{Abramovitz} or derived from the Stirling formula,  and reads as follows. For fixed $a>0$, $b\in \R$ it holds that
\begin{align}\label{GammaLimit1}
\Gamma(az + b) \sim (2 \pi)^{1/2}\, \exp(-az)\, (az)^{az+b-1/2}, \quad \text{as} \quad |z| \rightarrow \infty, \; |\arg z| < \pi -\eps.
\end{align}
From the moment formula in Theorem \ref{theo:vol_Simplices_Moments}(a) we obtain
\begin{align*}
\E e^{tn\cL_{n,r}} = (r+1)^{tn\over 2}2^{tnr\over 2}\prod_{j=1}^r{\Gamma\Big({(t+1)n-r+j\over 2}\Big)\over\Gamma\Big({n-r+j\over 2}\Big)}.
\end{align*}
Using \eqref{GammaLimit1} we deduce that
\begin{align}
\prod_{j=1}^r{\Gamma\Big({(t+1)n-r+j\over 2}\Big)\over\Gamma\Big({n-r+j\over 2}\Big)} &\sim \prod_{j=1}^r e^{-{tn\over 2}}\Big({n\over 2}\Big)^{tn\over 2}(t+1)^{{(t+1)n\over 2}+{j-r-1\over 2}}\notag \\
&=e^{-{tnr\over 2}}\Big({n\over 2}\Big)^{rtn\over 2}(t+1)^{\big({(t+1)n\over 2}-{1\over 2}\big)r-{r(r-1)\over 4}}. \label{eq:tech1}
\end{align}
Thus,
$$
\E e^{tn \cL_{n,r}} \sim (r+1)^{tn\over 2}e^{-{tnr\over 2}}n^{tnr\over 2}(t+1)^{\big({(t+1)n\over 2}-{1\over 2}\big)r-{r(r-1)\over 4}}.
$$
Taking the logarithm and subtracting $\frac r2 \log n$ and ${1\over 2}\log(r+1)$, we conclude that
\begin{align}\label{eq:GaussModelo(n)ModPhi}
\log \E \eee^{tn (\cL_{n,r}- \frac r2 \log n-{1\over 2}\log(r+1))} =
{nr\over 2}\Big((t+1)\log(t+1)-t\Big)-{r(r+1)\over 4}\log(t+1) + o(1)
\end{align}
and the result follows.
\end{proof}

\begin{remark}\label{rem:ModPhiGaussr=o(n)}
From the previous proof it easily follows that the asymptotic relation \eqref{eq:GaussModelo(n)ModPhi} is still valid if $r$ growths with $n$ in such a way that $r=o(n)$. This observation will be used below in the context of large deviation principles.
\end{remark}

\begin{proposition}\label{prop:BetaMomentGeneratingFunction}
Fix some $r\in\N$ and consider the Beta or the spherical model. Then, $n\cL_{n,r}$ converges in the mod-$\phi$ sense with
	$$
	\eta(t) = {(r+1)(t+1)\over 2}\log((r+1)(t+1))-{r(t+1)+1\over 2}\log(r(t+1)+1)-{t+1\over 2}\log(t+1)
	$$
	and parameter $w_n=n$, namely
	$$
	\lim_{n\to\infty}{\E e^{tn\cL_{n,r}}\over e^{n\eta(t)}} = (1+t)^{{1-\nu(r+1)\over 2}-{r(r-1)\over 4}}\bigg({(r+1)(t+1)\over r(t+1)+1}\bigg)^{\nu(r+1)-2r-1\over 2}
	$$
	uniformly as long as  $t$ stays in any compact subset of $\C\setminus(-\infty,-1)$.
\end{proposition}

\begin{proof}
	From the moment formula in Theorem \ref{theo:vol_Simplices_Moments}(b) we have
	\begin{align*}
	\E e^{tn\cL_{n,r}} = \prod_{j=1}^r\Bigg[{\Gamma\Big({n-r+j\over 2}+{tn\over 2}\Big)\over\Gamma\Big({n-r+j\over 2}\Big)}{\Gamma\Big({n+\nu\over 2}\Big)\over\Gamma\Big({n+\nu\over 2}+{tn\over 2}\Big)}\Bigg]{\Gamma\Big({n+\nu\over 2}\Big)\over\Gamma\Big({n+\nu\over 2}+{tn\over 2}\Big)}{\Gamma\Big({r(n+\nu-2)+(n+\nu)\over 2}+{(r+1)tn\over 2}\Big)\over\Gamma\Big({r(n+\nu-2)+(n+\nu)\over 2}+{rtn\over 2}\Big)}.
	\end{align*}
First of all, by \eqref{GammaLimit1},
	\begin{align*}
	{\Gamma\Big({n+\nu\over 2}\Big)\over\Gamma\Big({n+\nu\over 2}+{tn\over 2}\Big)} \sim (1+t)^{{1\over 2}-{\nu\over 2}-{(1+t)n\over 2}}\Big({n\over 2}\Big)^{-{tn\over 2}}\,e^{tn\over 2}.
	\end{align*}
It follows from \eqref{eq:tech1} that the first product in the moment formula asymptotically behaves like
	$$
	\prod_{j=1}^r\Bigg[{\Gamma\Big({n-r+j\over 2}+{tn\over 2}\Big)\over\Gamma\Big({n-r+j\over 2}\Big)}{\Gamma\Big({n+\nu\over 2}\Big)\over\Gamma\Big({n+\nu\over 2}+{tn\over 2}\Big)}\Bigg] \sim (1+t)^{-{r\nu\over 2}-{r(r-1)\over 4}}.
	$$
Again using \eqref{GammaLimit1}, we obtain
	\begin{align*}
	{\Gamma\Big({r(n+\nu-2)+(n+\nu)\over 2}+{(r+1)tn\over 2}\Big)\over\Gamma\Big({r(n+\nu-2)+(n+\nu)\over 2}+{rtn\over 2}\Big)}& \sim ((r+1)(t+1))^{{n(r+1)(t+1)\over 2}+{\nu(r+1)-2r-1\over 2}}\Big({n\over 2}\Big)^{{tn\over 2}}\,e^{-{tn\over 2}}\\
	&\qquad\qquad\times (r(t+1)+1)^{-{n(r(t+1)+1)\over 2}-{\nu(r+1)-2r-1\over 2}}\,.
	\end{align*}
	Thus, as $n\to\infty$, we get
	\begin{align*}
	\log \E e^{tn\cL_{n,r}} &= \bigg({1-\nu(r+1)\over 2}-{r(r-1)\over 4}-{(1+t)n\over 2}\bigg)\log(1+t) \\
	&\qquad +\bigg({n(r+1)(t+1)\over 2}+{\nu(r+1)-2r-1\over 2}\bigg)\log((r+1)(t+1))\\
	&\qquad -\bigg({n(r(t+1)+1)\over 2}+{\nu(r+1)-2r-1\over 2}\bigg)\log(r(t+1)+1) + o(1)
	\end{align*}
	and the result follows.
\end{proof}

\subsection{Mod-$\phi$ convergence for the ExpGamma distribution}
Many examples of mod-$\phi$ convergence are known in probability, number theory, statistical mechanics and random matrix theory. The most basic cases are probably the mod-Gaussian and mod-Poisson convergence, which can be found in \cite{jacod_etal,kowalski_nikeghbali,Feray1}, but there are also examples of mod-Cauchy~\cite{delbaen_etal,kowalski_najnudel_etal} and even mod-uniform \cite[\S 7.4]{Feray1} convergence. The aim of the present section is to add one more item to this list by proving a convergence modulo a tilted $1$-stable totally skewed distribution.

Let $X_n$ be a random variable having a Gamma distribution with parameters $(n,1)$, that is the probability density of $X_n$ is $\frac 1 {\Gamma(n)} x^{n-1} \eee^{-x}$, $x>0$. The distribution of $\log X_n$ is called the ExpGamma distribution. The probability density of $-\log X_n$ is given by
$$
\frac{1}{\Gamma(n)} \eee^{-\eee^{-x}}  \eee^{-x n}, \quad x\in\R,
$$
and is the limiting probability density of the $n$-th order upper order statistic in an i.i.d.\ sample of size $N\to\infty$ from the max-domain of attraction of the Gumbel distribution, or, equivalently, the density of the $n$-th upper order statistic in the Poisson point process with intensity $\eee^{-x}\dd x$, $x\in\R$; see~\cite[Theorem~2.2.2 on p.~33]{leadbetter_book}.
 It is easy to check that $\E \log X_n = \Gamma'(n)/\Gamma(n) =\psi(n)$ is the digamma function.
\begin{theorem}\label{theo:exp_gamma_mod_phi}
The sequence of random variables $n(\log X_n - \psi(n))$ converges in the mod-$\phi$ sense with $\eta(t) = (t+1)\log (t+1) - t$ and parameter $w_n=n$, namely
$$
\lim_{n\to\infty} \frac{\E \eee^{tn(\log X_n - \psi(n))}}{ \eee^{n ((t+1)\log (t+1) - t)}} = \frac {\eee^{t/2}}{ \sqrt{t+1}}
$$
uniformly as long as $t$ stays in any compact subset of  $\C \backslash (-\infty,-1)$.
\end{theorem}

\begin{proof}
By the properties of the Gamma distribution, we have
$$
\E \eee^{tn(\log X_n - \psi(n))} = \eee^{-tn \psi(n)} \E X_n^{tn} = \eee^{-tn \psi(n)} \frac{\Gamma(tn+n)}{\Gamma(n)}.
$$
The Stirling formula states that $\Gamma(z) \sim \sqrt{2\pi/z}\, (z/\eee)^z$ uniformly as $|z|\to\infty$ in such a way that $|\arg z| < \pi -\eps$. Using the Stirling formula together with  the  asymptotics $\psi(n) = \log n - \frac 1{2n} + o(\frac 1n)$, we obtain
$$
\eee^{-tn \psi(n)} \frac{\Gamma(tn+n)}{\Gamma(n)}
\sim
\eee^{-tn(\log n - \frac 1 {2n})}
\frac{\sqrt{\frac {2\pi}{tn+n}}\left(\frac{tn+n}{\eee}\right)^n}{\sqrt{\frac {2\pi}{n}}\left(\frac{n}{\eee}\right)^n}
= \frac{\eee^{t/2}}{\sqrt{t+1}} \eee^{n ((t+1)\log (t+1) - t)},
$$
which proves the claim.
\end{proof}

\begin{remark}
Consider an $\alpha$-stable random variable $Z_1\sim S_1 (\pi/2, -1,0)$ with $\alpha=1$, skewness $\beta= -1$, and scale $\sigma= \pi/2$, where we adopt the parametrization used in the book of~\cite{samorodnitsky_taqqu_book}. It is known~\cite[Proposition 1.2.12]{samorodnitsky_taqqu_book} that the cumulant generating function of this random variable is given by
$$
\log \E \eee^{t Z_1} =   t \log t, \quad  \Re t \geq 0.
$$
Note that $\E \eee^Z_1 = 1$ and consider an exponential tilt of $Z_1$, denoted  $Z_2$,  whose probability  density is
$$
\P[Z_2 \in \dd x] = \eee^x \P[Z_1\in \dd x], \quad x\in\R.
$$
Finally, observe that $\E Z_2 = \E [\eee^{Z_1} Z_1] = (t^t)'|_{t=1} = 1$ and consider the centered version $Z:= Z_2-1$. The cumulant generating function of $Z$ is given by
$$
\log \E \eee^{t Z} =   (t+1) \log (t+1) - t, \quad \Re t\geq -1.
$$
As an exponential tilt of an infinitely divisible distribution, $Z$ is itself infinitely divisible. Thus, in Theorem~\ref{theo:exp_gamma_mod_phi} and Proposition~\ref{prop:mod_phi_constant_r} we have a mod-$\phi$ convergence modulo a tilted totally skewed $1$-stable distribution.
\end{remark}

\subsection{Mod-$\phi$ convergence in the full dimensional case}
In this section we consider the full-dimensional case $r=n$, i.e., we are interested in the random variable $\cL_{n,n}$.
\begin{proposition}\label{prop:mod_phi_full_dim}
Consider the Gaussian model and let $m_n = \frac 12 (n \log n - n + \frac 12 \log n  + \log (2^{3/2}\pi))$. Then, $\cL_{n,n} - m_n$ converges in the mod-Gaussian sense (meaning that $\eta(t) = \frac 12 t^2$)  and parameters $w_n = \frac 12 \log \frac n2$, namely
$$
\lim_{n\to\infty} \frac{\E \eee^{t(\cL_{n,n} - m_n)}}{\eee^{\frac 14 t^2\log \frac n2}} = \frac{G\left(\frac{1}{2}\right)}{G\left(\frac{1}{2} + \frac t2\right) G\left(1 + \frac t2\right)}.
$$
The convergence is uniform as long as $t$ stays in any compact subset of $\C\backslash\{-1,-2,\ldots\}$.
\end{proposition}
\begin{proof}
In view of Theorem~\ref{theo:vol_Simplices_Moments}(a) and~\eqref{eq:prod_gamma_functions}, we can express the moment generating function of $\cL_{n,n}$ in terms of the Barnes $G$-function as
\begin{equation}\label{eq:moment_gen_Y_n_n}
\E \eee^{t\cL_{n,n}}
= \E [(n! \cV_{n,n})^t]
= (n+1)^{\frac{t}{2}}\, 2^{\frac{tn}{2}}\,
\frac{G\left(\frac{1}{2}\right)}{G\left(\frac{n+1}{2}\right)}
\cdot
\frac{G(1)}{G\left(\frac{n+2}{2}\right)}
\cdot
\frac{G\left(\frac{n+1}{2} + \frac t2\right)}{G\left(\frac{1}{2} + \frac t2\right)}
\cdot
\frac{G\left(\frac{n+2}{2} + \frac t2\right)}{G\left(1 + \frac t2\right)},
\end{equation}
where $G(1) = 1$. For the function
\begin{equation}\label{eq:PsiDef}
\psi(t) := \frac{G\left(\frac{1}{2}\right)}{G\left(\frac{1}{2} + \frac t2\right) G\left(1 + \frac t2\right)}
\end{equation}
we have
\begin{align*}
\log \E \eee^{t\cL_{n,n}}
&= \frac{t}{2}\log(n+1) + \frac{tn}{2} \log 2 +
 \log \psi(t) + \log G\left(\frac{n-1}{2} +\frac t2 +1\right) - \log G\left(\frac{n-1}{2} +1\right)\\
&\qquad+ \log G\left(\frac{n}{2} +\frac t2 +1\right) - \log G\left(\frac{n}{2} +1\right).
\end{align*}
Applying Lemma~\ref{lem:barnes_G_diff_asympt} two times
and using the formula
\begin{equation*}
((n+b) \log (n+b) - (n+b)) - (n\log n - n) = b \log n + o(1),
\end{equation*}
where $b$ is any constant, we obtain
\begin{equation}\label{eq:GaussFallModPhiFull}
\log \E \eee^{t\cL_{n,n}}
=
\log \psi(t)
+
\frac t2 \left(n \log n - n + \frac 12 \log n  + \log (2^{3/2}\pi)\right) +\frac 14 {t^2} \log \frac n2 +o(1).
\end{equation}
This completes the argument.
\end{proof}

\begin{remark}\label{rm:DalBorgoEtAl}
We notice that in the full dimensional, Gaussian case $r=n$ our random variables are equivalent to those considered in \cite{dalBorgo_etal} and one can follow our result also from their Theorem 5.1. Nevertheless, we decided to include our independent and much shorter proof.\\
Their paper deals with the determinant of certain random matrix models and has a completely different focus. On the other hand, let us emphasize that even in this special case the distributions appearing in \cite{dalBorgo_etal} are in fact different from (but very close to) those we obtain.
\end{remark}

\begin{proposition}\label{prop:mod_phi_full_dimBeta}
Consider the Beta model with parameter $\nu>0$ or the spherical model (in which case $\nu=0$) and let $\widetilde{m}_n =  {1\over 2}({1\over 2}\log n-n+1-\nu+\log(2^{3/2}\pi))$. Then, $\cL_{n,n} - \widetilde{m}_n$ converges in the mod-Gaussian sense (meaning that $\eta(t) = \frac 12 t^2$)  and parameters $w_n = \frac 12 \log \frac n2-\frac 12$, namely
$$
\lim_{n\to\infty} \frac{\E \eee^{t(\cL_{n,n} - \widetilde{m}_n)}}{\eee^{\frac 14 t^2(\log \frac n2-1)}} = \frac{G\left(\frac{1}{2}\right)}{G\left(\frac{1}{2} + \frac t2\right) G\left(1 + \frac t2\right)}.
$$
The convergence is uniform as long as $t$ stays in any compact subset of $\C\backslash\{-1,-2,\ldots\}$.
\end{proposition}
\begin{proof}
For the purposes of this proof let $\cL_{n,n}^{\text{G}}$ denote the Gaussian analogue of $\cL_{n,n}$. In view of the connection between the Gaussian and the Beta model, see Theorem~\ref{theo:vol_Simplices_Moments}(a),(b),  the moment generating function of $\cL_{n,n}$  is given by
\begin{multline*}
\log \E \eee^{t\cL_{n,n}}
=
\log \E \eee^{t\cL_{n,n}^{\text{G}}} - \frac t2 \log (n+1) - \frac {tn}{2} \log 2
\\
+ (n+1)\log\left({\Gamma\Big({n+\nu\over 2}\Big)\over\Gamma\Big({n\over 2}+{\nu + t\over 2}\Big)}\right)
+ \log \left({\Gamma\Big({n(n+\nu-1)+nt+ t + \nu \over 2}\Big)\over\Gamma\Big({n(n+\nu-1)+nt +\nu\over 2}\Big)}\right).
\end{multline*}
Using a second-order Stirling approximation for the logarithms of the Gamma functions, we obtain
\begin{align*}
(n+1)\log\left({\Gamma\Big({n+\nu\over 2}\Big)\over\Gamma\Big({n\over 2}+{\nu + t\over 2}\Big)}\right) = {(n+1)t\over 2}\log{2\over n}-{t\over 4}(t-2+2\nu) + o(1)
\end{align*}
and similarly
\begin{align*}
\log \left({\Gamma\Big({n(n+\nu-1)+nt+ t + \nu \over 2}\Big)\over\Gamma\Big({n(n+\nu-1)+nt +\nu\over 2}\Big)}\right)
 = t\log n-{t\over 2}\log 2+o(1).
\end{align*}
Denoting by $\psi(t)$ the function defined at \eqref{eq:PsiDef} and using \eqref{eq:GaussFallModPhiFull} we conclude that, after simplification of the resulting terms,
\begin{align*}
\log \E e^{t\cL_{n,n}} = \log\psi(t)+t\widetilde{m}_n + {t^2\over 4}\Big(\log{n\over 2}-1\Big) +o(1)
\end{align*}
from which the result follows.
\end{proof}

\subsection{Case of fixed codimension}

Consider the case in which the codimension of the simplex $n-r$ stays fixed, while $n\to\infty$. Of course, if $n-r=0$, we recover the full-dimensional case.
\begin{proposition}\label{prop:mod_phi_full_dim_almost}
Consider the Gaussian model and let $m_n$ be the same as in Proposition~\ref{prop:mod_phi_full_dim}. Let $d\in \N$ be fixed and take $r=n-d$, where $n\to\infty$. Then, $\cL_{n,r} - m_n$ converges in the mod-Gaussian sense (meaning that $\eta(t) = \frac 12 t^2$) with parameters $w_n = \frac 12 \log \frac n2$, namely
$$
\lim_{n\to\infty} \frac{\E \eee^{t(\cL_{n,r} - m_n)}}{\eee^{\frac 14 t^2\log \frac n2}}
=
\frac {G\left(\frac{d+1}{2}\right) \cdot G\left(\frac{d+2}{2}\right)}
{2^{\frac{td}{2}}\,
G\left(\frac{d+1}{2} + \frac t2\right)
\cdot
G\left(\frac{d+2}{2} + \frac t2\right)
}
.
$$
The convergence is uniform as long as $t$ stays in any compact subset of $\C\backslash\{-d-1,-d-2,\ldots\}$.
\end{proposition}
\begin{proof}
First, we observe that Theorem \ref{theo:vol_distr_affine} implies the distributional representation
\begin{align}\label{eq:DistributionalIdentityCL}
\cL_{n,n}-{1\over 2}\log(n+1) \overset{d}{=} \Big(\cL_{n-r,n-r}-{1\over 2}\log(n-r+1)\Big)+\Big(\cL_{n,r}'-{1\over 2}\log(r+1)\Big),
\end{align}
where $\cL_{n,r}'$ is a copy of $\cL_{n,r}$ independent of $\cL_{n-r,n-r}$. Since $n-r=d$, this implies that
\begin{align*}
\E e^{t(\cL_{n,r}-m_n)} = \frac{\E \eee^{t(\cL_{n,n} - m_n)}}{\E \eee^{t \cL_{d,d}}}\,e^{{t\over 2}\log(d+1)}\,e^{{t\over 2}\log\big({n-d+1\over n+1}\big)}.
\end{align*}
Applying Proposition~\ref{prop:mod_phi_full_dim} to the numerator and~\eqref{eq:moment_gen_Y_n_n} to the denominator, we conclude that
$$
\E \eee^{t(\cL_{n,r} - m_n)}
\sim
\frac {\eee^{\frac 14 t^2\log \frac n2} \frac{G\left(\frac{1}{2}\right)}{G\left(\frac{1}{2} + \frac t2\right) G\left(1 + \frac t2\right)}}
{(d+1)^{\frac{t}{2}}\, 2^{\frac{td}{2}}\,   \frac{G\left(\frac{1}{2}\right)}{G\left(\frac{d+1}{2}\right)}
\cdot
\frac{G(1)}{G\left(\frac{d+2}{2}\right)}
\cdot
\frac{G\left(\frac{d+1}{2} + \frac t2\right)}{G\left(\frac{1}{2} + \frac t2\right)}
\cdot
\frac{G\left(\frac{d+2}{2} + \frac t2\right)}{G\left(1 + \frac t2\right)}
}\cdot (d+1)^{t\over 2}\,,
$$
which implies the claim.
\end{proof}

\begin{proposition}\label{propositionbeta}
Consider the Beta model with parameter $\nu>0$ or the spherical model (in which case $\nu=0$) and let $\tilde m_n$ be the same as in Proposition~\ref{prop:mod_phi_full_dimBeta}.
Let $d\in \N$ be fixed and take $r=n-d$, where $n\to\infty$. Then, $\cL_{n,r} - \tilde{m}_n - \frac{d-1}{2} \log \frac n2$ converges in the mod-Gaussian sense (meaning that $\eta(t) = \frac 12 t^2$)  and parameters $w_n = \frac 12 \log \frac n2-\frac 12$, namely
$$
\lim_{n\to\infty} \frac{\E \eee^{t(\cL_{n,r} - \widetilde{m}_n - \frac{d-1}{2} \log \frac n2)}}{\eee^{\frac 14 t^2(\log \frac n2-1)}} =
\frac {G\left(\frac{d+1}{2}\right) \cdot G\left(\frac{d+2}{2}\right)}
{2^{\frac{td}{2}}\,
G\left(\frac{d+1}{2} + \frac t2\right)
\cdot
G\left(\frac{d+2}{2} + \frac t2\right)
}.
$$
The convergence is uniform as long as $t$ stays in any compact subset of $\C\backslash\{-d-1,-d-2,\ldots\}$.
\end{proposition}
\begin{proof}
The computations are similar to those in the proof of Proposition~\ref{prop:mod_phi_full_dimBeta}, but slightly more involved.
Again we let $\cL_{n,r}^{\text{G}}$ to be the Gaussian analogue of $\cL_{n,r}$. By Theorem~\ref{theo:vol_Simplices_Moments}(a),(b),  the moment generating function of $\cL_{n,r}$  is given by
\begin{multline}\label{eq:L_n_n_gauss_vs_beta}
\log \E \eee^{t\cL_{n,r}}
=
\log \E \eee^{t\cL_{n,r}^{\text{G}}} - \frac t2 \log (r+1) - \frac {tr}{2} \log 2
\\
+ (r+1)\log\left({\Gamma\Big({n+\nu\over 2}\Big)\over\Gamma\Big({n\over 2}+{\nu + t\over 2}\Big)}\right)
+ \log \left({\Gamma\Big({r(n+\nu-2)+n+\nu\over 2}+{(r+1)t\over 2}\Big)\over\Gamma\Big({r(n+\nu-2)+n+\nu\over 2}+{rt\over 2}\Big)}\right).
\end{multline}
Using the Stirling series for the logarithm of the Gamma function, we obtain
$$
(r+1)\log\left({\Gamma\Big({n+\nu\over 2}\Big)\over\Gamma\Big({n\over 2}+{\nu + t\over 2}\Big)}\right)  = {(n+1)t\over 2}\log{2\over n}-{t\over 4}(t-2+2\nu) + \frac{d-1}{2} t \log \frac n2 + o(1)
$$
and
$$
\log \left({\Gamma\Big({r(n+\nu-2)+n+\nu\over 2}+{(r+1)t\over 2}\Big)\over\Gamma\Big({r(n+\nu-2)+n+\nu\over 2}+{rt\over 2}\Big)}\right)
=
 t\log n-{t\over 2}\log 2+o(1).
$$
Using the behavior of $\cL_{n,r}^{\text{G}}$ stated in Proposition~\ref{prop:mod_phi_full_dim_almost}, we obtain, after some transformations,
\begin{align*}
&\log \E \eee^{t\cL_{n,r}}\\
&\qquad =
\log\left(
\frac {G\left(\frac{d+1}{2}\right) \cdot G\left(\frac{d+2}{2}\right)}
{2^{\frac{td}{2}}\,
G\left(\frac{d+1}{2} + \frac t2\right)
\cdot
G\left(\frac{d+2}{2} + \frac t2\right)}\right)
+ t\widetilde{m}_n + {t^2\over 4}\Big(\log{n\over 2}-1\Big) + \frac{d-1}{2} t \log \frac n2+ o(1),
\end{align*}
which yields the claim.
\end{proof}

\subsection{Case of diverging codimension}

In this section we consider the case  when the codimension of the simplex goes to $+\infty$.

\begin{proposition}\label{prop:mod_phi_div_codimension}
Consider the Gaussian model and let $m_n$ be the same as in Proposition~\ref{prop:mod_phi_full_dim}.
If $r=r(n)$ is such that $n-r \to  \infty$ as $n\to\infty$, then 
$$
\lim_{n\to\infty}
\frac{\E \eee^{t\big(\cL_{n,r} - (m_n-m_{n-r})-{1\over 2}\log\big({(r+1)(n-r)\over n}\big)\big)}}{\eee^{\frac 14 t^2 \log \frac{n}{n-r}}} = 1.
$$
If, additionally,  $n-r=o(n)$, then we have mod-Gaussian convergence (meaning that $\eta(t)={1\over 2}t^2$) with parameters $w_n = \frac 1 2 \log \frac {n}{n-r} \to\infty$ and limiting function identically equal to $1$.
\end{proposition}
\begin{proof}
From Proposition~\ref{prop:mod_phi_full_dim} we know that
$$
\lim_{n\to\infty} \frac{\E \eee^{t(\cL_{n,n} - m_n)}}{\eee^{\frac 14 t^2\log \frac n2}}
=
\lim_{n\to\infty} \frac{\E \eee^{t(\cL_{n-r,n-r} - m_{n-r})}}{\eee^{\frac 14 t^2\log \frac {n-r}2}}
=
\frac{G\left(\frac{1}{2}\right)}{G\left(\frac{1}{2} + \frac t2\right) G\left(1 + \frac t2\right)}.
$$
Using the distributional identity \eqref{eq:DistributionalIdentityCL} it follows that
\begin{align*}
\E \eee^{t\big(\cL_{n,r} - (m_n-m_{n-r})-{1\over 2}\log\big({(r+1)(n-r+1)\over n+1}\big)\big)} =  \frac{\E \eee^{t(\cL_{n,n} - m_n)}}{\E \eee^{t(\cL_{n-r,n-r} - m_{n-r})}}
\sim
\frac {\eee^{\frac 14 t^2\log \frac n2}}{\eee^{\frac 14 t^2\log \frac {n-r}2}}
\sim
\eee^{\frac 14 t^2\log \frac n{n-r}},
\end{align*}
which implies the claim after observing that $\log (n+1) = \log n + o(1)$ and $\log (n-r+1) = \log (n-r)+o(1)$. Observe also that if $n-r=o(n)$, then $w_n\to\infty$, as $n\to\infty$, which is otherwise not the case.
\end{proof}


\begin{proposition}\label{propositionbeta2}
Consider the Beta model with parameter $\nu>0$ or the spherical model (in which case $\nu=0$) and let $m_n$ be the same as in Proposition~\ref{prop:mod_phi_full_dim}.  If $r=r(n)$ is such that $n-r=o(n)$ as $n\to\infty$,  then,
$$
\lim_{n\to\infty}
\frac{\E \eee^{t\big(\cL_{n,r} - (m_n-m_{n-r}-{r+1\over 4n}(t-2+2\nu))- \frac 12 \log \frac{(n-r)(1+r)}{n^{1+r}}\big)}}{\eee^{\frac 14 t^2\log{n\over n-r}}} = 1.
$$
That is, we have mod-Gaussian convergence (meaning that $\eta(t)={1\over 2}t^2$) with parameters $w_n={1\over 2}\log{n\over n-r}$ and limiting function identically equal to $1$.
\end{proposition}
\begin{proof}
Denote by $\cL_{n,r}^{\text{G}}$ the Gaussian analogue of $\cL_{n,r}$. Observe that relation~\eqref{eq:L_n_n_gauss_vs_beta} still holds. Regarding the first term in this relation, we know from Proposition~\ref{prop:mod_phi_div_codimension} that
$$
\log \E \eee^{t\cL_{n,r}^{\text{G}}}
=  t(m_n-m_{n-r}) + {t\over 2}\log\left({(r+1)(n-r)\over n}\right)
+\frac 14 t^2 \log \frac{n}{n-r} + o(1).
$$
Again, a second-order Stirling expansion yields
$$
(r+1)\log\left({\Gamma\Big({n+\nu\over 2}\Big)\over\Gamma\Big({n\over 2}+{\nu + t\over 2}\Big)}\right)
=(r+1) \frac t2 \log \frac 2n -{r+1\over n}{t\over 4}(t-2+2\nu)+o(r/n^2)
$$
and
\begin{align*}
\log \left({\Gamma\Big({r(n+\nu-2)+n+\nu\over 2}+{(r+1)t\over 2}\Big)\over\Gamma\Big({r(n+\nu-2)+n+\nu\over 2}+{rt\over 2}\Big)}\right)
&=
\frac t2 \log \left({r(n+\nu-2)+n+\nu\over 2}+{rt\over 2}\right)+o(1)\\
&=
\frac t2 \log \left(\frac {(r+1)n}{2}\right)+o(1).
\end{align*}
Taking everything together, we obtain
\begin{align*}
&\log \E \eee^{t\cL_{n,r}}\\
&\qquad=
t\Big(m_n-m_{n-r}-{r+1\over 4n}(t-2+2\nu)\Big) + \frac t2 \log \frac{(n-r)(1+r)}{n^{1+r}}
+
\frac 14 t^2 \log \left(\frac n {n-r}\right) + o(1).
\end{align*}
This yields the claim, since $w_n={1\over 2}\log{n\over n-r}\to\infty$, as $n\to\infty$ by the assumption that $n-r=o(n)$.
\end{proof}

\section{Large deviations}

The purpose of this section is to derive large deviation principles, recall the definition at the beginning of Section \ref{SectionLDP}. Again, we restrict to the Gaussian, the Beta and the spherical model, which admit finite moments of all orders.

\subsection{The Gaussian model}

We start with the Gaussian model and recall the notation $\cL_{n,r} := \log(r! \cV_{n,r})$. Using the G\"artner--Ellis theorem we derive large deviation principles from the following Proposition.

\begin{proposition}\label{prop:gaertner_ellis_cond}
\begin{itemize}
\item[(a)] Let  $r=o(n)$, as $n\to\infty$. Then, we have
$$
\lim_{n\to\infty} \frac 1 {rn} \log \E \eee^{tn (\cL_{n,r}-  \frac r2 \log n - \frac{1}{2} \log (r+1))} = \frac 12 ((t+1)\log (t+1) - t).
$$
\item[(b)] If $r\sim \alpha n$, $\alpha \in (0,1)$, we have
\begin{align*}
\lim_{n\to\infty} \frac 1 {\alpha n^2} \log \E \eee^{tn (\cL_{n,r}-  \frac{\alpha n}{2} (\log n + \log(1-\alpha)))}    = \frac{2+ 2t - \alpha}{4} \log\left(\frac{1 + t - \alpha}{1 - \alpha}\right) - \frac{t}{2}.
\end{align*}
\item[(c)] Let $d\in \N$ and assume that $d = n - r$, as $n \rightarrow \infty$, and $m_n = \frac 12 (n \log n - n + \frac 12 \log n  + \log (2^{3/2}\pi))$ as in Proposition \ref{prop:mod_phi_full_dim_almost}. Then, we have
\begin{align*}
\lim_{n\to\infty} \frac 1 {\frac 12 \log  \frac{n}{2}} \log \E \eee^{t (\cL_{n,r}- m_n)} = \frac 12 t^2.
\end{align*}
\item[(d)] Let $r=r(n)$ be such that $n-r = o(n)$, as $n\rightarrow \infty$. Then, we have
\begin{align*}
\lim_{n\to\infty} \frac 1 {\frac 12 \log \frac{n}{n-r}} \log \E \eee^{t \big(\cL_{n,r}- (m_n - m_{n-r})-{1\over}\log\big({(r+1)(n-r+1)\over n+1}\big)\big)} = \frac 12 t^2.
\end{align*}
\end{itemize}
\end{proposition}

\begin{proof}
Part (a) is a consequence of Remark \ref{rem:ModPhiGaussr=o(n)}. The proofs of the (c) and (d) directly follow from the proofs of Propositions \ref{prop:mod_phi_full_dim_almost} and \ref{prop:mod_phi_div_codimension} in the previous section, respectively.

We turn now to the case that $r\sim\alpha n$. Due to the asymptotic formula \eqref{GammaLimit1} we obtain for all $\alpha \in (0,1)$, $t\ge 0$ and $j\in \mathbb{N}$ that
\begin{align*}
\log\left( {\Gamma\big({(1+t - \alpha)n+j\over 2}\big)\over\Gamma\big({(1-\alpha)n+j\over 2}\big)}  \right) &\sim \log\left(\frac{\exp(-\frac{(1+t-\alpha)n}{2}) \left(\frac{(1+t-\alpha)n}{2}\right)^{\frac{(1+t-\alpha)n}{2} + \frac{j-1}{2}}}{\exp(-\frac{(1-\alpha)n}{2}) \left(\frac{(1-\alpha)n}{2}\right)^{\frac{(1-\alpha)n}{2} + \frac{j-1}{2}}}\right)\\
&=\log\left( \exp\left(-\frac{t n}{2}\right) \left(\frac n2 \right)^{\frac{tn}{2}} \frac{\left(1+t - \alpha\right)^{\frac{(1+t-\alpha)n}{2}}}{\left(1-\alpha\right)^{\frac{(1-\alpha)n}{2}}}  \left(\frac{1+t-\alpha}{1-\alpha}\right)^{\frac{j-1}{2}} \right)\\
&= -\frac{t n}{2} + \frac{tn}{2} \log\left(\frac n2\right) + \frac{(1+t-\alpha)n}{2} \log\left(1+t-\alpha\right)\\
&\qquad \qquad - \frac{(1-\alpha)n}{2} \log\left(1-\alpha\right) + \frac{j-1}{2} \log\left( \frac{1+t-\alpha}{1-\alpha}\right),
\end{align*}
as $n \rightarrow \infty$, and thus
\begin{align*}
&\frac{1}{\alpha n^2} \log \E \eee^{tn \cL_{n,r}}  = \frac{1}{\alpha n^2} \left[\frac{tn}{2}\log (\alpha n + 1) + \frac{t \alpha n^2}{2} \log 2 + \sum_{j=1}^{\alpha n} \log\left( {\Gamma\big({(1+t-\alpha)n+j\over 2}\big)\over\Gamma\big({(1-\alpha)n+j\over 2}\big)}  \right)\right]\\
&\sim -\frac{t}{2} + \frac{t}{2} \log\left(n\right) + \frac{1+t-\alpha}{2} \log\left(1+t-\alpha\right) - \frac{1-\alpha}{2} \log\left(1-\alpha\right) + \frac{\alpha}{4} \log\left( \frac{1+t-\alpha}{1-\alpha}\right)\\
&= -\frac{t}{2} + \frac{t}{2} \log\left(n\right) + \frac{2+2t-\alpha}{4} \log\left(1+t-\alpha\right) - \frac{2-\alpha}{4} \log\left(1-\alpha\right)\,.
\end{align*}
This directly yields the result in the case $r \sim \alpha n$ in view of the moment formula for Gaussian simplices stated in Section \ref{sec:SectionModels}.
\end{proof}

We turn now to the large deviation principles for the log-volume of Gaussian simplices.

\begin{theorem}[LDP for Gaussian simplices]\label{LDPGaus}
\begin{itemize}
\item[(a)] Let $r=o(n)$, as $n\to\infty$. Then, ${1\over r}(\cL_{n,r}- \frac r2 \log n - \frac{1}{2} \log (r+1))$ satisfies a LDP with speed $rn$ and rate function
$$
I(x) = \frac 12 (e^{2x}-1) - x\,,\qquad x\in\R\,.
$$
\item[(b)] If $r\sim  \alpha n$, $\alpha \in (0,1)$, then, ${1\over \alpha n}(\cL_{n,r}- \frac{\alpha n}{2} (\log n + \log(1-\alpha)))$ satisfies a LDP with speed $\alpha n^2$ and rate function
$$
I(x) = \sup_{t \in \mathbb{R}} \left\{t x - \frac{2+ 2t - \alpha}{4} \log \left(\frac{1 + t - \alpha}{1 - \alpha}\right) +  \frac{t}{2}\right\} \,,\qquad x\in\R\,.
$$
\item[(c)] Let $d\in \N$ and assume that $d = n - r$, as $n \rightarrow \infty$, and $m_n = \frac 12 (n \log n - n + \frac 12 \log n  + \log (2^{3/2}\pi))$. Then, $\frac{1}{\frac{1}{2} \log \frac{n}{2}}(\cL_{n,r}- m_n)$ satisfies a LDP with speed $\frac{1}{2} \log \frac{n}{2}$ and rate function
$$
I(x) = \frac 12 x^2\,,\qquad x\in\R\,.
$$
\item[(d)] Let $r=r(n)$ be such that $n-r \rightarrow \infty$, as $n \rightarrow \infty$. If $n-r = o(n)$, as $n\rightarrow \infty$, then, $\frac{1}{\frac{1}{2} \log \frac{n}{n-r}}\big(\cL_{n,r}-(m_n - m_{n-r})-{1\over 2}\log\big({(r+1)(n-r+1)\over n+1}\big)\big)$ satisfies a LDP with speed $\frac{1}{2} \log \frac{n}{n-r}$ and rate function
$$
I(x) = \frac 12 x^2\,,\qquad x\in\R\,.
$$
\end{itemize}
\end{theorem}
\begin{proof}[Proof of Theorem \ref{LDPGaus}]
Let $r=o(n)$, as $n\to\infty$. Then, by the G\"artner--Ellis theorem (cf.\ Section 2.3 in \cite{Dembo}) and Proposition \ref{prop:gaertner_ellis_cond}, the random variables ${1\over r}(\cL_{n,r}-\frac r2 \log n - \frac{1}{2} \log (r+1))$ satisfy a LDP with speed $rn$ and rate function
$$
I(x) = \sup_{t\in\R}\big[tx-\frac 12 ((t+1)\log (t+1) - t)\big],
$$
i.e., the Legendre-Fenchel transform of the function $\frac 12 ((t+1)\log (t+1) - t)$. For each $x\in\R$ the supremum is attained at $t=e^{2x}-1$, which yields the result of (a). The same argument implies the LDP for the other regimes of $r$ as well. 
\end{proof}

\subsection{The Beta and the spherical model}
Now, we turn to the Beta model with parameter $\nu > 0$ and the spherical model, i.e., $\nu = 0$, and recall that $\cL_{n,r}:=\log(r!\cV_{n,r})$, where $\cV_{n,r}$ is the volume of the $r$-dimensional simplex with vertices $X_1,\ldots,X_{r+1}$ chosen according to the Beta or the spherical distribution, respectively. Similar to the Gaussian case, we start with the following proposition that will imply the large deviation principles.

\begin{proposition}\label{prop:gaertner_ellis_cond2}
\begin{itemize}
\item[(a)] 	Let  $r\in\N$ be fixed. Then, we have
	$$
	\lim_{n\to\infty} \frac 1 {n} \log \E \eee^{tn \cL_{n,r}} = \begin{cases}
	\eta(t) &: t\geq -1\\
	+\infty & \text{otherwise},
	\end{cases}
	$$
	where $\eta$ is the function from Proposition \ref{prop:BetaMomentGeneratingFunction}.
\item[(b)] 	If $r \sim \alpha n$, $\alpha \in (0,1)$, we have
	\begin{align*}
	&\lim_{n\to\infty} \frac 1 {\alpha n^2} \log \E \eee^{tn\cL_{n,r}}  = \begin{cases}
	\eta(t) &: t\geq -1\\
	+\infty &: \text{otherwise},
	\end{cases}
	\end{align*}
		where $\psi(t)$ is the function given by
		$$
		\eta(t):=\frac{2+2t-\alpha}{4} \log\left(1+t-\alpha\right) - \frac{2-\alpha}{4} \log\left(1-\alpha\right)- \frac{1+t}{2} \log(1+t)\,.
		$$
\item[(c)] 	Let $d\in \N$ and assume that $d = n - r$, as $n \rightarrow \infty$, and let $\widetilde{m}_n =  {1\over 2}({1\over 2}\log n-n+1-\nu+\log(2^{3/2}\pi))$ as in Proposition \ref{prop:mod_phi_full_dimBeta}. Then, we have
	\begin{align*}
	\lim_{n\to\infty} \frac 1 {\frac 12 (\log  \frac{n}{2} - 1)} \log \E \eee^{t (\cL_{n,r}- \widetilde{m}_n - \frac{d-1}{2} \log \frac{n}{2})} = \frac{1}{2} t^2.
	\end{align*}
\item[(d)] 	Let $r=r(n)$ be such that $n-r = o(n)$, and let $m_n = \frac 12 (n \log n - n + \frac 12 \log n  + \log (2^{3/2}\pi))$ be as in Proposition \ref{prop:mod_phi_full_dim_almost}. Then
	\begin{align*}
	\lim_{n\to\infty} \frac 1 {\frac 12 \log \frac{n}{n-r}} \log \E \eee^{t\big(\cL_{n,r} - (m_n-m_{n-r}-{r+1\over 4n}(t-2+2\nu))- \frac 12 \log \frac{(n-r)(1+r)}{n^{1+r}}\big)} = {1\over 2}t^2.
	\end{align*}
\end{itemize}
\end{proposition}

\begin{proof}[Proof of Proposition \ref{prop:gaertner_ellis_cond}]
	For $t\geq -1$ the assertions in (a) follows from Proposition \ref{prop:BetaMomentGeneratingFunction}. Recall from Theorem \ref{theo:vol_distr_affine} that the distribution of $\cV_{n,r}$ involves Beta random variables $Z:=\beta_{{\nu+r-j\over 2},{n-r+j\over 2}}$ with $j\leq r$. Writing
	$$
	\E e^{{tn\over 2}\log Z} = \E Z^{tn\over 2} = c\,\int_0^{1} z^{{n-r+j\over 2}+{tn\over 2}-1}(1-z)^{{\nu+r-j\over 2}-1}\,\dd z
	$$
	we see that the exponent at $z$ is less than $-1$ for sufficiently large $n$ if $t<-1$. This implies that $\E e^{{tn\over 2}\log Z}\to+\infty$ and completes the proof of (a).
	
	Now, let us turn towards the case $r\sim\alpha n$, $\alpha \in (0,1)$ in (b).
	Similar to what has been done in the Gaussian setting, we obtain by using the asymptotic formula \eqref{GammaLimit1} for all $\nu > 0$,
	\begin{align*}
	(\alpha n+1) \log\left( {\Gamma\big({n+\nu\over 2}\big)\over\Gamma\big({(1+t)n+\nu\over 2}\big)}  \right) &\sim (\alpha n + 1) \left(\frac{t n}{2} - \frac{tn}{2} \log\left(\frac n2\right) - \frac{(1+t)n + \nu -1}{2} \log(1+t) \right)\\
	&\sim \frac{t \alpha n^2}{2} - \frac{t\alpha n^2}{2} \log\left(\frac n2\right) - \frac{(1+t)\alpha n^2 + \alpha n (\nu -1)}{2} \log(1+t),
	\end{align*}
	as $n\rightarrow \infty$, and for all $t\ge 0$,
	\begin{align*}
	&\log\left( {\Gamma\big({\alpha n(n+\nu-2) + n + tn(\alpha n +1) +\nu\over 2}\big)\over\Gamma\big({\alpha n(n+\nu-2) + n + tn\alpha n +\nu\over 2}\big)}  \right)\\
	&\sim -\frac{t  n}{2} + \frac{tn}{2} \log\left(\frac n2\right) + \frac{\alpha n(n+\nu-2) + n + tn(\alpha n +1) +\nu}{2} \log\left(\alpha (n+\nu-2) + 1 + t(\alpha n +1)\right)\\
	&\qquad \qquad - \frac{\alpha n(n+\nu-2) + n + tn\alpha n  +\nu}{2} \log\left(\alpha (n+\nu-2) + 1+ t\alpha n\right).
	\end{align*}
	Thus, by using the calculations made in the Gaussian case above, we conclude that
	\begin{align*}
	&\frac{1}{\alpha n^2} \log \E \eee^{tn\cL_{n,r}}\\
	&= \frac{1}{\alpha n^2}\left[ (\alpha n+1) \log\left( {\Gamma\big({n+\nu\over 2}\big)\over\Gamma\big({(1+t)n+\nu\over 2}\big)}  \right)\right.\\
	&\left.\qquad + \log\left( {\Gamma\big({\alpha n(n+\nu-2) + n + tn(\alpha n +1) +\nu\over 2}\big)\over\Gamma\big({\alpha n(n+\nu-2) + n + tn\alpha n +\nu\over 2}\big)}  \right) + \sum_{j=1}^{\alpha n} \log\left( {\Gamma\big({(1+t-\alpha)n+j\over 2}\big)\over\Gamma\big({(1-\alpha)n+j\over 2}\big)}  \right)\right]\\
	&\sim \frac{t}{2} - \frac{t}{2} \log\left(\frac n2\right) - \frac{1+t}{2} \log(1+t) + \frac{1+t}{2}  \log\left(\alpha (n+\nu-2) + 1 + t(\alpha n +1)\right)\\
	&\qquad - \frac{1+t}{2}  \log\left(\alpha (n+\nu-2) + 1 + t\alpha n\right) -\frac{t}{2} + \frac{t}{2} \log\left(\frac n2\right) + \frac{2+2t-\alpha}{4} \log\left(1+t-\alpha\right)\\
	&\qquad - \frac{2-\alpha}{4} \log\left(1-\alpha\right)\\
	&\sim - \frac{1+t}{2} \log(1+t) + \frac{2+2t-\alpha}{4} \log\left(1+t-\alpha\right) - \frac{2-\alpha}{4} \log\left(1-\alpha\right),
	\end{align*}
	as $n\rightarrow \infty$.
	This directly yields the result in the case where $r \sim  \alpha n$, again taking into account the moment representation in the Beta model stated in Section \ref{sec:SectionModels}. Since there is no dependence on the parameter $\nu$ in the result concerning the Beta model, the one regarding the spherical model is implied by considering the limiting case $\nu \downarrow 0$ as seen several times before.\\
	The proofs of the (c) and (d) directly follow from the proofs of Propositions \ref{propositionbeta} and \ref{propositionbeta2} in the previous section, respectively.
\end{proof}

Now, we are able to state the large deviation principles for the Beta and the spherical model. Their proofs follow the same lines as the ones in the Gaussian case presented above by using the G\"artner--Ellis theorem. For this reason we have decided to skip them.

\begin{theorem}[LDP for Beta-type and spherical simplices]\label{LDPBeta}
\begin{itemize}
\item[(a)] 	Let $r\in\N$ be fixed. Then, $\cL_{n,r}$ satisfies a LDP with speed $n$ and rate function
	$$
	I(x) = \sup_{t\in\R}\big\{tx-\eta(t)\big\},
	$$
	where $\eta$ is the function from Proposition \ref{prop:BetaMomentGeneratingFunction}.
\item[(b)] 	If $r\sim \alpha n$, $\alpha \in (0,1)$, then, ${1\over \alpha n}\cL_{n,r}$ satisfies a LDP with speed $\alpha n^2$ and rate function
	$$
	I(x) = \sup_{t \in\R} \left\{t x - \eta(t) \right\} ,
	$$
	where $\eta$ is the function from Proposition \ref{prop:gaertner_ellis_cond2} (b).
\item[(c)] 	Let $d\in \N$ and assume that $d = n - r$, as $n \rightarrow \infty$, and $\widetilde{m}_n =  {1\over 2}({1\over 2}\log n-n+1-\nu+\log(2^{3/2}\pi))$. Then, $\frac{1}{\frac{1}{2}(\log \frac{n}{2} - 1)}(\cL_{n,r}- \widetilde{m}_n - \frac{d-1}{2} \log \frac{n}{2})$ satisfies a LDP with speed $\frac{1}{2}(\log \frac{n}{2} - 1)$ and rate function
	$$
	I(x) = \frac{1}{2} x^2 \,,\qquad x\in\R\,.
	$$
\item[(d)] 	Let $r=r(n)$ be such that $n-r = o(n)$, and let $m_n = \frac 12 (n \log n - n + \frac 12 \log n  + \log (2^{3/2}\pi))$ be defined as in Proposition \ref{prop:mod_phi_full_dim_almost}. Then, ${1\over {1\over 2}\log{n\over n-r}}\big(\cL_{n,r} - (m_n-m_{n-r}-{r+1\over 4n}(t-2+2\nu))- \frac 12 \log \frac{(n-r)(1+r)}{n^{1+r}}\big)$ satisfies a LDP with speed ${1\over 2}\log{n\over n-r}$ and rate function
	$$
	I(x) = {1\over 2}x^2 \,,\qquad x\in\R\,.
	$$
\end{itemize}
\end{theorem}

\begin{remark}
One can combine Theorem \ref{LDPBeta} with the contraction principle from large deviation theory to obtain a LDP for $\cV_{n,r}$, that is, for the volume of the random simplex itself in the cases that $r=o(n)$ and $r\sim \alpha n$ for some $\alpha\in(0,1)$.
\end{remark}

\subsection*{Acknowledgement}
JG has been supported by the German Research Foundation (DFG) via Research Training Group RTG 2131 \textit{High dimensional Phenomena in Probability -- Fluctuations and Discontinuity}. ZK and CT were supported by the DFG Scientific Network \textit{Cumulants, Concentration and Superconcentration}.

\end{document}